\documentclass[10pt,twoside,english,reqno,a4paper]{amsart}

\usepackage{listings,graphicx,amsmath,varioref,amscd,amssymb,color,bm,stmaryrd,amsthm,amsfonts,graphics,geometry,latexsym,pgf,pst-all} 
\theoremstyle{plain}
\usepackage{esint}
\usepackage{amsthm}
\usepackage{tikz}
\usepackage{pgfplots}
\usepackage{enumerate}

\usepackage{resmes} 
\usepackage{graphicx}
\usepackage{caption}
\usepackage{subcaption}
\theoremstyle{plain}
\newtheorem{theorem}{Theorem}[section]
\newtheorem{proposition}[theorem]{Proposition}

\newtheorem{lemma}[theorem]{Lemma}

\newtheorem{corollary}[theorem]{Corollary}

\usepackage{geometry}
\usepackage[colorlinks=false]{hyperref}

\theoremstyle{definition}
\newtheorem{defin}[theorem]{Definition}

\newtheorem{remark}[theorem]{Remark}
\newtheorem{example}{Example}

\theoremstyle{remark}

\usepackage{fouriernc}
\usepackage[T1]{fontenc}

\numberwithin{equation}{section}

\def\dis{\displaystyle}

\DeclareMathOperator{\diver}{div}

\DeclareMathOperator{\R}{\mathbb{R}}
\DeclareMathOperator{\N}{\mathbb{N}}

\newcommand{\car}[1]{\raise1pt\hbox{$\chi$}_{#1}}

\newcommand{\DM }{\mathcal{DM}^\infty }
\def\rn{\mathbb{R}^N}
\def\into{\int_{\Omega}}
\def\intdo{\int_{\partial\Omega}}
\def\dH{\mathrm{d}\mathcal{H}^{N-1}}
\def\ae{\mathrm{a.e.}}

\def\bvo{{BV(\Omega)}}
\def\lio{{L^\infty(\Omega)}}
\def\lno{{L^N(\Omega)}}
\def\sob{{W_0^{1,p}(\Omega)}}

\begin{document}
\title[]{Existence and  classification of solutions for singular 1-Laplacian equations with gradient terms}

\author[A. J. Mart\'inez Aparicio]{Antonio J. Mart\'inez Aparicio}
\author[F. Oliva]{Francescantonio Oliva}
\author[F. Petitta]{Francesco Petitta}

\address[Antonio J. Mart\'inez Aparicio]{Departamento de Matem\'aticas,
Universidad de Almer\'ia
	\hfill \break\indent
    Ctra. Sacramento s/n, La Ca\^{n}ada de San Urbano, 04120 Almer\'ia, Spain}
\email{\tt ajmaparicio@ual.es}

\address[Francescantonio Oliva]{
Dipartimento di Scienze di Base e Applicate per l' Ingegneria, Sapienza Universit\`a di Roma
	\hfill \break\indent
	Via Scarpa 16, 00161 Roma, Italy}
\email{\tt francescantonio.oliva@uniroma1.it}
\address[Francesco Petitta]{Dipartimento di Scienze di Base e Applicate per l' Ingegneria, Sapienza Universit\`a di Roma
		\hfill \break\indent
		Via Scarpa 16, 00161 Roma, Italy}
\email{\tt francesco.petitta@uniroma1.it}


\begin{abstract}
We study the existence and qualitative properties of non-negative solutions to singular 1-Laplacian problems of the form
\begin{equation*}
	\begin{cases}
		-\Delta_1 u = g(u)|D u| + h(u)f & \text{in } \Omega,\\
		u=0 & \text{on } \partial\Omega,
	\end{cases}
\end{equation*}
where $\Omega\subset \mathbb{R}^N$ is a bounded domain with Lipschitz boundary, $0<f\in L^{N,\infty}(\Omega)$, and $g,h$ are continuous positive functions that may blow up at the origin. We establish the existence of a positive bounded solution even when $g$ is non-integrable at zero and $h(0)=+\infty$, breaking the non-existence phenomenon appearing in  the associated $p$-Laplacian problems. Furthermore, we prove that solutions possess a non-zero boundary trace on a set of positive measure of $\partial\Omega$ whenever $g(s)$ grows sufficiently fast near $s=0$; by contrast, when this growth condition fails, solutions vanishing on the boundary may appear. Finally, when $h(0)<+\infty$,  we provide a sharp threshold  classification  of trivial and non-trivial solutions depending on the size of $f$.
\end{abstract}

\maketitle

\setcounter{tocdepth}{1}
\tableofcontents

\section{Introduction}

The problem we address in this article is
\begin{equation}
\label{eq:PbIntro1}
	\begin{cases}
		-\Delta_1 u = g(u)|Du| + h(u)f & \text{in }\Omega,\\
		u=0 & \text{on } \partial\Omega,
	\end{cases}
\end{equation}
where $\Delta_1 u = \operatorname{div}(|Du|^{-1} Du)$ is the 1-Laplace operator, $\Omega\subset \R^N$  ($N\ge 2$) is a bounded domain with Lipschitz boundary, and $0<f(x)\in L^{N,\infty}(\Omega)$. We are particularly interested in the case where $g$ and $h$ are continuous functions possibly singular at the origin. As model functions, one might consider $g(s)=\lambda s^{-\theta}$ and $h(s)=s^{-\gamma}$, where $\theta,\gamma\geq 0$ and $\lambda>0$.

This problem presents several difficulties. On the one hand, it involves a singular gradient term with natural growth, as well as a singular zero-order term. Besides, the problem is driven by the 1-Laplacian, a highly degenerate operator.

For the $p$-Laplacian ($p>1$), problems involving a gradient term with natural growth have been widely investigated. There are many existence and non-existence results concerning the problem
\begin{equation}
\label{eq:PbIntro2}
	\begin{cases}
		-\Delta_p u = g(u)|\nabla u|^p + h(u)f & \text{in }\Omega,\\
		u=0 & \text{on } \partial\Omega,
	\end{cases}
\end{equation}
under different conditions on $f$, $g$ and $h$. When problem~\eqref{eq:PbIntro2} is non-singular (for instance, when $g\in C(\R)$ and $h\equiv 1$) one finds a vast number of contributions; we refer, without being exhaustive, to the classical works \cite{AbDaPe, BMP1, FeMu, PoSe}. As in the non-singular  case, when $g$ blows up at the origin, the nature of  the problem depends on whether the gradient is an absorption or a reaction term. Some works dealing with the absorption case ($g(s)\approx -s^{-\theta}$) are~\cite{ACLMOP, BO, dura, pedro, xz}. Regarding the reaction case ($g(s)\approx s^{-\theta}$), the threshold $\theta=1$ is critical for the existence of global finite energy solutions, at least when $f$ is sufficiently integrable. This can be seen through the following heuristic argument. Consider the problem
\begin{equation}
\label{eq:PbIntro3}
	\begin{cases}
		-\Delta_p u = \lambda\frac{|\nabla u|^p}{u^\theta} + \frac{f}{u^\gamma} & \text{in }\Omega,\\
		u=0 & \text{on } \partial\Omega,
	\end{cases}
\end{equation}
where $\theta, \lambda, f>0$ and $\gamma\geq 0$. If $\theta=1$, one may formally use $u \geq 0$ ($u\not\equiv 0$) as test function in~\eqref{eq:PbIntro3} to obtain that
\[
\left(1-\lambda \right) \into |\nabla u|^p = \into fu^{1-\gamma} > 0.
\]
Therefore, if $\lambda\geq 1$, no $\sob$-solution to~\eqref{eq:PbIntro3} can exist. When $\theta>1$, a similar argument shows that for any $\lambda>0$ the non-existence phenomenon persists. For a formal proof of these results, see~\cite{Ol}. On the other hand, several existence results are available when $\theta<1$ and $\lambda>0$, or when $\theta=1$ and $\lambda<1$, see for instance~\cite{ABLP,gps1,gps2}. 

In this work, we study whether these thresholds of existence and non-existence are also present when considering the 1-Laplacian operator. It should be noted that this operator presents some specific  difficulties with respect to the $p$-Laplacian. First, the lack of compactness of $W^{1,1}(\Omega)$ makes $BV(\Omega)$ the natural space for finding the solutions. Recall that $BV(\Omega)$ is the space of the $L^1$ functions $u$ whose distributional gradient $Du$ is a Radon measure with finite total variation. Since $Du$ is just a Radon measure, a second difficulty arises when giving meaning to the quotient $|Du|^{-1} Du$ appearing in the definition of $\Delta_1 u$. In the seminal paper \cite{AnBaCaMa}, the authors employed Anzellotti's theory (\cite{Anz}) to define a concept of solution for 1-Laplace problems, which is now widely adopted, and which involves a bounded vector field $z$ which plays the role of that quotient.

A striking consequence of this definition is the way in which the Dirichlet boundary condition is understood, allowing the existence of solutions with non-zero boundary trace. Indeed, for $u$ to be a solution of a 1-Laplace problem, one simply requires on the boundary that $[z,\nu](x) \in \operatorname{sign}(-u(x))$ for $\mathcal{H}^{N-1}$-$\ae$ $x\in \partial\Omega$. Here, $[z,\nu](x)$ denotes the generalized normal trace of $z$ at $x\in\partial\Omega$.

There is a vast literature on problem~\eqref{eq:PbIntro1} when $g\equiv 0$. Unlike what occurs with the $p$-Laplacian, the existence of solution to
\begin{equation}
\label{eq:PbIntro4}
	\begin{cases}
		-\Delta_1 u = h(u)f & \text{in }\Omega,\\
		u=0 & \text{on } \partial\Omega,
	\end{cases}
\end{equation}
with $h\equiv 1$ and $f\in L^{N,\infty}(\Omega) \subset W^{-1,\infty}(\Omega)$ strongly depends on the size of $f$. If $\mathcal{\tilde S}_1$ denotes the best constant of the Sobolev embedding $W_0^{1,1}(\Omega) \hookrightarrow L^{\frac{N}{N-1},1}(\Omega)$, then a $BV$-solution to~\eqref{eq:PbIntro4} exists if $\|f\|_{L^{N,\infty}(\Omega)} \leq \mathcal{\tilde S}_1^{-1}$, whereas it does not exist when $\|f\|_{W^{-1,\infty}(\Omega)}>1$ (\cite{MeSeTr}, see also~\cite{CiTr, Kaw}). Surprisingly, when $\|f\|_{L^{N,\infty}(\Omega)} < \mathcal{\tilde S}_1^{-1}$ or $\|f\|_{W^{-1,\infty}(\Omega)} < 1$ the unique solution is $u\equiv 0$, and non-zero solutions can only appear when these smallness conditions are not satisfied. In this latter case, the structure of the 1-Laplacian leads to a strong nonuniqueness phenomenon: if $u$ is a solution to~\eqref{eq:PbIntro4} (with $h\equiv 1$), then any smooth increasing transformation of $u$ preserving the origin is also a solution. For instance, if $\Omega$ is the unit ball and $f\equiv N$ (see that $\|f\|_{L^{N,\infty}(\Omega)} \geq \mathcal{\tilde S}_1^{-1}$ and $\|f\|_{W^{-1,\infty}(\Omega)} = 1$), then any $u\equiv \alpha>0$ solves~\eqref{eq:PbIntro4}.

Introducing a nonconstant function $h$ in~\eqref{eq:PbIntro4} induces some regularizing effects, at least when $h$ is bounded at infinity. For simplicity, assume that $f>0$ and $h\colon \R^+\to \R^+$ is continuous. If one denotes $h(\infty):=\limsup_{s\to +\infty} h(s)<+\infty$, then a bounded non-negative $BV$-solution to~\eqref{eq:PbIntro4} exists if $\|f\|_{L^{N,\infty}(\Omega)} < (\mathcal{\tilde S}_1 h(\infty) )^{-1}$. Therefore, if $h(\infty)=0$, no smallness is required in $f$. In addition, this global $BV$-solution is positive in $\Omega$ (and hence non-trivial) provided that $h$ is singular at the origin, that is, $h(0):=\lim_{s\to 0^+} h(s) = +\infty$. These results can be found in \cite{MaOlPe1} (see also~\cite{DecGOP, LaSeOlPe}).

When $h$ is unbounded at infinity, the existence of solution depends on the specific structure of $h$. Some works dealing with these nonlinearities are \cite{MaOlPe2, MoSe}.

Another property of solutions to~\eqref{eq:PbIntro4} is that, as they belong to $BV(\Omega)$, they may have jumps. Indeed, this really occurs in some cases (see, for instance, \cite[Section~6.3]{LaSeOlPe}). The situation is entirely different when a gradient term is introduced, as in problem
\begin{equation}
\label{eq:PbIntro5}
	\begin{cases}
		-\Delta_1 u + g(u)|Du| = h(u)f & \text{in }\Omega,\\
		u=0 & \text{on } \partial\Omega.
	\end{cases}
\end{equation}
Under certain assumptions on $g$ and $h$, a global $BV$ solution to~\eqref{eq:PbIntro5} without jump part exists for any positive $f\in L^1(\Omega)$. Furthermore, this solution has zero boundary trace. More specifically, it is sufficient for $g$ to be integrable near zero and non-vanishing at infinity, while $h$ needs only to satisfy that $h(\infty)<+\infty$ (both functions are allowed to blow up at zero). When $h$ is non-increasing, the solution is unique. These results are contained in~\cite{AJMA}, see also the earlier works~\cite{Bal, BalOP, MaSe, LaSe1, LaSe2}.

Of all these good properties, the reaction gradient case only preserves the property that solutions have no jump part. If one takes $g\equiv 1$ and $h\equiv 1$ in problem
\begin{equation}
\label{eq:PbIntro6}
	\begin{cases}
		-\Delta_1 u = g(u)|Du| + h(u)f & \text{in }\Omega,\\
		u=0 & \text{on } \partial\Omega,
	\end{cases}
\end{equation}
then a strong multiplicity phenomenon arises as long as $0\leq f\in L^m(\Omega)$ $(m>N)$ is sufficiently small; indeed, infinitely many $BV$-solutions exist (\cite{AbDaSe}). By contrast, there are some functions $f\in L^N(\Omega)$ for which a $BV$-solution does not exist (\cite{AnDaSe, DaSe}). Regarding the boundary condition, some solutions have zero boundary trace, while others have non-zero trace. 

More general functions are considered in~\cite{GOP}, where the authors take $0<f\in L^{N,\infty}(\Omega)$ and $g,h\in C(\R^+,\R^+)$, with $h(s)\approx s^{-\gamma}$ ($\gamma\geq 0$) near the origin and $h(\infty)<+\infty$. When the gradient term is non-singular or slightly singular (that is, when $g(0)<+\infty$ or $g(s)\approx s^{-\theta}$ ($\theta<1$) near zero), they prove the existence of a bounded non-negative $BV$-solution to~\eqref{eq:PbIntro6} with no jump part provided that $\|f\|_{L^{N,\infty}(\Omega)} < (\mathcal{\tilde S}_1 h(\infty) )^{-1}$. Here, they take advantage of the integrability of $g$ near the origin.

The case where $g$ is non-integrable at the origin is considerably more complicated, and different strategies must be employed to handle the term $g(u)|Du|$. In continuity with the $p$-Laplacian setting (see~\eqref{eq:PbIntro3}), for $g$ modelled as $g(s)= \frac{\lambda}{s}$ with $\lambda>0$, the authors in~\cite{GOP} establish an existence result for every $\lambda<1$, provided that the function $f$ satisfies
\begin{equation}
    \label{eq:Intro_f}
    \forall \omega \subset\subset \Omega \ \ \exists c_\omega>0: f\geq c_\omega \ \ae \text{ in } \omega,
\end{equation}
and the function $h$ is singular and behaves as $s^{-\gamma}$ ($\gamma> 0$) near zero.

In this work, we significantly generalize the result in ~\cite{GOP}  by showing  that the restriction $\lambda<1$ is merely technical, and that existence actually holds for any $\lambda>0$. Moreover, we also address the case of stronger singularities, modelled as $g(s)=\frac{\lambda}{s^\theta}$, $\lambda>0$, $\theta\geq 1$, and we prove the existence of a bounded $BV$-solution to~\eqref{eq:PbIntro6} provided that $f$ verifies~\eqref{eq:Intro_f} and $h(s) \approx s^{-\gamma}$ ($\gamma> 0$) near zero.

Our existence result reveals a continuity-breaking phenomenon with respect to the $p$-Laplacian framework (for a related continuity-breaking phenomenon when $g\equiv 0$, see~\cite{MaOlPe1}). Therefore, to establish this, we need to employ a novel approach that exploits the specific structural features of the 1-Laplace operator in the presence of a gradient lower order term.

Additionally, we show that a very singular reaction gradient term has an anti-regularizing effect on the boundary value of the solutions to~\eqref{eq:PbIntro6}. More precisely, we prove that if $g$ is non-integrable near zero, then no solution $u$ to~\eqref{eq:PbIntro6} has zero boundary trace. To our knowledge, this is the first general result that states that any solution to a whole family of problems is positive somewhere on $\partial\Omega$. This contrasts sharply with the case of absorption gradient terms, where solutions always have zero boundary trace.

To better illustrate our results, we state them for the model problem (see Theorems~\ref{th:Exist} and~\ref{th:Trace}). 

\begin{theorem}
Let $0<f\in L^{N,\infty}(\Omega)$, and let $\lambda,\theta,\gamma>0$. Then the problem
\begin{equation}
\label{eq:PbIntro7}
	\begin{cases}
		-\Delta_1 u = \lambda \frac{|Du|}{u^\theta} + \frac{f}{u^\gamma} & \text{in }\Omega,\\
		u=0 & \text{on } \partial\Omega,
	\end{cases}
\end{equation}
has a positive solution $u\in \bvo\cap \lio$ if one of the following cases holds:
\begin{enumerate}
    \item[(i)] $\theta<1$,
    \item[(ii)] $\theta\geq 1$ and $f$ satisfies~\eqref{eq:Intro_f}.
\end{enumerate}
Furthermore, if $\theta\geq 1$, then for any solution $u\in \bvo$ to~\eqref{eq:PbIntro7} there exists a set $S \subset \partial\Omega$ with $\mathcal{H}^{N-1}(S)>0$ such that
\[
u(x) > 0,\ \forall x\in S.
\]
\end{theorem}

The definition of solution itself forces any solution to~\eqref{eq:PbIntro6} to be positive in $\Omega$, provided that $h$ is singular at the origin. Another question investigated in this paper is whether solutions to~\eqref{eq:PbIntro6} are trivial (i.e., identically zero) or non-trivial when $h(0)<+\infty$. As we show in Theorem~\ref{th:Dual}, this depends entirely on the size of $f$: when $f$ is small, problem~\eqref{eq:PbIntro6} admits only the trivial solution, whereas when $f$ is large, every solution is non-trivial. Between these sharp thresholds, we prove the existence of a ``transition regime'' where trivial and non-trivial solutions can coexist. This constitutes one of the main novelties of our work, extending and complementing the previous results in~\cite{CiTr, DecGOP, MeSeTr, Kaw}.

Finally, we emphasize that all the results in this paper are complemented by several examples illustrating the sharpness of our assumptions.

\subsection{Plan of the paper}

The rest of the paper is organized as follows. In Section~\ref{sec:Prelim}, we collect the necessary background to be used throughout the manuscript. After establishing basic notation, we review the fundamental properties of functions of bounded variation, as well as Anzellotti's theory of $L^\infty$-measure-divergence vector fields. We also prove a technical result that ensures, under suitable conditions, that a $BV$ function has no jump part.

In Section~\ref{sec:Main}, we state our main existence (Theorem~\ref{th:Exist}) and trace (Theorem~\ref{th:Trace}) results. Section~\ref{sec:Exist} is then devoted to the proof of Theorem~\ref{th:Exist}, which relies on an approximation scheme; uniform estimates are derived via a bootstrap argument, while the limiting process requires a novel technique to bypass singularities. In Section~\ref{sec:Trace}, we prove Theorem~\ref{th:Trace}, showing that no solution can have a zero boundary trace whenever the gradient term is non-integrable at the origin. The sharpness of the main results is exhibited through several explicit examples in Section~\ref{sec:Sharpness}.

Section~\ref{sec:trivial_solutions} establishes sharp thresholds for solution triviality and non-triviality when the zero-order term is non-singular. Between these thresholds, we illustrate with examples that a ``transition regime'' emerges, where the problem may simultaneously admit both trivial and non-trivial solutions.

Finally, for the reader's convenience, we include Appendix~\ref{sec:Appendix}, where we prove the existence of bounded solutions for the non-singular approximate problems used in Section~\ref{sec:Exist}.

\section{Preliminaries}
\label{sec:Prelim}

\subsection{Notation}

Throughout this article, $\Omega$ is an open bounded subset of $\R^N$ ($N\geq 2$) with Lipschitz boundary. Given a set $E\subset \R^N$, $\mathcal{H}^{N-1}(\partial E)$ stands for the $(N-1)$-dimensional Hausdorff measure of its boundary, while $|E|$ denotes the classical $N$-dimensional Lebesgue measure $\mathcal{L}^N(E)$.

We say that $f_1<f_2$ in $E$ if $f_1(x)<f_2(x)$ for almost every $x\in E$. As for the integrals, we denote by $\into f$ the integral of $f$ with respect to the Lebesgue measure, whereas $\into f\mu$ denotes the integral of $f$ with respect to the measure $\mu$.

For a fixed $k>0$, we set $T_{k}$ and $G_{k}$ as the real functions defined for any $s\in \R$ as
\begin{equation}
\label{eq:def_Tk_Gk}
T_k(s):=\max \{-k,\min \{s,k\} \} \qquad \text{and} \qquad G_k(s):= \min\{s+k, \max\{0,s-k\}\}. 
\end{equation}
Moreover, we represent the characteristic function of a set $A\subseteq\Omega$ by $\chi_A$.

Finally, we denote by $\mathcal{S}_1$ and $\tilde{\mathcal{S}}_1$ the best constants of the embeddings $W_0^{1,1}(\Omega) \hookrightarrow L^{\frac{N}{N-1}}(\Omega)$ and $W_0^{1,1}(\Omega) \hookrightarrow L^{\frac{N}{N-1},1}(\Omega)$, respectively (see~\eqref{eq:emb1} and~\eqref{eq:emb2}). Their values are well known and given by $\mathcal{S}_1 = \left( N \omega_{N}^{1/N} \right)^{-1}$ and $\tilde{\mathcal{S}}_1 = \left( (N-1)\omega_{N}^{1/N} \right)^{-1}$.

\subsection{Functions of bounded variation} We denote by $\mathcal{M}(\Omega)$ the space of Radon measures with finite total variation over $\Omega$. The space of functions of bounded  variation is defined as
\[
BV(\Omega):= \left\{ u\in L^1(\Omega) : Du \in \mathcal{M}(\Omega)^N \right\},
\]
where $Du$ stands for the distributional gradient of $u$. Moreover, we denote by $BV_{\rm loc} (\Omega)$ the set of functions $u\in L^1_{\rm loc}(\Omega)$ that belong to $BV(\omega)$ for any open set $\omega$ compactly contained in $\Omega$.

Every function in $\bvo$ has a trace in $L^1(\partial\Omega)$. This allows to define the norm
\[
\|u\|_{BV(\Omega)}=\int_\Omega|Du| + \int_{\partial\Omega} |u| \ \dH,\ \forall u\in\bvo,
\]
where $\int_\Omega|Du|$ denotes the total variation of the measure $Du$ over $\Omega$. It is worth noting that $\|\cdot\|_\bvo$ is equivalent to the norm given by $\|u\|=\int_\Omega|Du| + \int_\Omega |u|$. Furthermore, $\|\cdot\|_\bvo$ is   lower semicontinuous with respect to the $L^1(\Omega)$ convergence. The same property holds for the functional $u\mapsto \int_\Omega \varphi|Du|$ with $0\leq \varphi\in C(\overline\Omega)$ fixed. 

Recall that for $u \in L_{\rm loc}^1(\Omega)$, we say that $u$ has an approximate limit at $x_0 \in \Omega$ if there exists $z\in \R$ such that 
\begin{equation*}
    \lim_{r\to 0^+} \frac{1}{|B_r(x_0)|}\int_{B_r(x_0)} |u(y) - z| \, \mathrm{d}y = 0.
\end{equation*}
The set $S_u$ of points where this property does not hold is called the approximate discontinuity set. This set $S_u$ is a $\mathcal{L}^N$-negligible Borel set (\cite[Proposition~3.64]{AFP}). For any $x_0\in \Omega\setminus S_u$ the number $z$, which is uniquely determined, is called the approximate limit of $u$ at $x_0$ and is denoted by $\widetilde{u}(x_0)=z$.

On the other hand, we say that $x_0\in\Omega$ is an approximate jump point of $u$ if there exist $a,b\in \R$ and a unit vector $\nu \in S^{N-1}$ such that $a\neq b$ and
\begin{align*}
    \lim_{r\to 0^+} \frac{1}{|B^+_r(x_0)|} \int_{B^+_{r}(x_0)} |u(y) - a| \, \mathrm{d}y &= 0, \\
    \lim_{r\to 0^+} \frac{1}{|B^-_r(x_0)|} \int_{B^-_{r}(x_0)} |u(y) - b| \, \mathrm{d}y &= 0,
\end{align*}
where $B_r^\pm(x_0) := \{y\in B_r(x_0): \pm (y-x_0)\cdot \nu >0\}$. The triplet $(a,b,\nu)$, uniquely determined up to a permutation of $(a,b)$ and a change of sign of $\nu$, is denoted by $(u^+(x_0), u^-(x_0), \nu_u(x_0))$. The set of approximate jump points is denoted by $J_u$. It is a Borel subset of $S_u$ (\cite[Proposition~3.69]{AFP}), and $\mathcal{H}^{N-1}(S_u \setminus J_u) = 0$ if $u \in BV(\Omega)$. Moreover, up to an $\mathcal{H}^{N-1}$-negligible set, $J_u$ is an $\mathcal{H}^{N-1}$-rectifiable set, and an orientation $\nu_u(x)$ is defined for $\mathcal{H}^{N-1}$-almost every $x \in J_u$ (\cite[Theorem~3.78]{AFP}).

The precise representative of $u$  is defined as the $\mathcal{H}^{N-1}$-a.e. finite function 
$u^*\colon \Omega\setminus(S_u\setminus J_u)\to \mathbb{R}$ given by 
\begin{equation*}
u^*(x)=
\begin{cases}
\widetilde{u}(x)& \text{ if } x\in \Omega\setminus S_u,\\
\frac{u^+(x)+u^-(x)}{2}& \text{ if } x\in J_u.
\end{cases}
\end{equation*}
In the particular case where $\mathcal{H}^{N-1}(J_u)=0$, we write $u$ instead of $u^*$ for the precise representative of $u$, as no ambiguity is possible when integrating against a measure which is absolutely continuous with respect to $\mathcal{H}^{N-1}$.

If $u\in\bvo$, then $Du$ can be decomposed as
\[
Du = D^a u + D^s u,
\]
where $D^a u$ is the absolutely continuous part of $Du$ with respect to the Lebesgue measure $\mathcal{L}^N$, and $D^s u$ is the singular part. The measure $D^a u$ is just $\nabla u \, \mathcal{L}^N$, where $\nabla u$ is the approximate gradient of $u$ (\cite[Definition~3.70]{AFP}). The singular part can be further decomposed as the sum of its Cantor and jump part as
\[
D^s u = D^c u + D^j u, \qquad D^c u := D^s u \llcorner (\Omega\setminus S_u), \qquad D^j u:= D^s u \llcorner J_u = (u^+ - u^-)\nu_u \mathcal{H}^{N-1} \llcorner J_u.
\]
We denote by $\widetilde D u = D^a u + D^c u$ the diffuse part of $Du$. An important feature of $\widetilde D u$ is that, for every $t\in\R$, it holds that $\widetilde D u \llcorner \{u=t\} = 0$ (see \cite[Proposition~3.92]{AFP}).

A chain rule formula for these functions is available.

\begin{lemma}{\cite[Theorem 3.99]{AFP}}
\label{lem:ChainRule_AFP}
Let $u\in BV_{\rm loc}(\Omega)$. If $\Psi \colon \R\to \R$ is a Lipschitz function, then $\Psi (u)\in BV_{\rm loc}(\Omega)$ and
\begin{equation}
\label{eq:chain_rule_1}
D\Psi(u) = \Psi'(u) \widetilde D u + (\Psi(u^+) - \Psi(u^-)) \nu_u \mathcal{H}^{N-1} \llcorner J_u.
\end{equation}
\end{lemma}

\begin{remark}
The chain rule can be rewritten in a more concise form. With this aim, given a locally integrable function $\beta\colon \R\to\R$ and $v\in BV_{\rm loc}(\Omega)$, we set 
\begin{equation}
\label{eq:def_beta_hash}
\beta(v)^\# := \begin{cases}
    \frac{1}{v^+ - v^-} \int_{v^-}^{v^+} \beta(s) \, \mathrm{d}s, & \text{ if } x\in J_v,\\
    \beta(\tilde v), & \text{if } x\in \Omega\setminus S_v.
\end{cases}
\end{equation}
Observe that $\beta(v)^\#$ is a particular representative of $\beta(v)$ and coincides with $\beta(v)^*$ if $\beta$ is an affine function, but in general $\beta(v)^\# \neq \beta(v)^*$. By means of this function, the chain rule~\eqref{eq:chain_rule_1} can be expressed as
\begin{equation}
\label{eq:chain_rule_2}
D\Psi(u) = \Psi'(u)^\# D u.
\end{equation}
\end{remark}

In this work, we mainly deal with functions without a jump part. In such cases, a more general version of the chain rule holds, provided that the composition is known to belong to $BV_{\rm loc}(\Omega)$.

\begin{lemma}{\cite[Lemma 2.3]{AJMA}}
\label{lem:ChainRule}
Let $0\leq u\in BV_{\rm loc}(\Omega)$ be such that $D^j u = 0$, and let $g\colon [0,+\infty) \to (0,+\infty]$ be a continuous function finite outside the origin. Let $\Gamma\colon [0,+\infty) \to [-\infty, +\infty)$ be such that $\Gamma'(s) = g(s)$, and suppose that $\Gamma(u)\in BV_{\rm loc}(\Omega)$. Then $g(u) \in L^1_{\rm loc}(\Omega, |Du|)$ and it holds
\[
\chi_{\{u>0\}}|D\Gamma(u)| = g(u) |Du| \text{ as measures in } \Omega.
\]
Consequently, if $\Gamma(0)$ is finite, then
\[
|D\Gamma(u)| = g(u) |Du| \text{ as measures in } \Omega.
\]
\end{lemma}

For further properties of functions of bounded variation, we refer the reader to~\cite{AFP}. 

\subsubsection{Sobolev embeddings and Lorentz spaces}

The classic embedding $W_0^{1,1}(\Omega) \hookrightarrow L^{\frac{N}{N-1}}(\Omega)$ can be extended to $BV(\Omega)$ functions by approximation (\cite{AFP}), leading to the inequality
\begin{equation}
\label{eq:emb1}
\|u\|_{L^{\frac{N}{N-1}}(\Omega)} \leq \mathcal{S}_1 \|u\|_{BV(\Omega)},\ \forall u\in \bvo.
\end{equation}
Here, the optimal constant is $\mathcal{S}_1=\left(N\omega_N^{1/N} \right)^{-1}$, where $\omega_{N}$ denotes the volume of the unit ball in $\rn$ (see, for instance,~\cite{Lorentz}).

In the context of Lorentz spaces, the above embedding can be refined, as $\bvo$ is continuously embedded into $L^{\frac{N}{N-1},1}(\Omega)$, which is strictly smaller than $L^{\frac{N}{N-1}}(\Omega)$. Consequently,
\begin{equation}
\label{eq:emb2}
\|u\|_{L^{\frac{N}{N-1},1}(\Omega)} \leq \tilde{\mathcal{S}}_1 \|u\|_{BV(\Omega)},\ \forall u\in \bvo,
\end{equation}
where $\tilde{\mathcal{S}}_1 = \left( (N-1)\omega_{N}^{1/N} \right)^{-1}$ is the optimal constant (see~\cite{Lorentz}).

Throughout this work, the use of Lorentz spaces is restricted to the above embedding and basic inequalities. Recall that a measurable function $v\colon \Omega\to \R$ is said to belong to the Lorentz space $L^{p,q}(\Omega)$, with $1\leq p < +\infty$ and $1\leq q\leq +\infty$, if the following quasinorm is finite:
\begin{equation*}
\|v\|_{L^{p,q}(\Omega)}=\dis \begin{cases}
    \dis \left(\int_{0}^{+\infty} \left(t^\frac{1}{p} v_*(t)\right)^q\frac{\,dt}{t}\right)^{\frac{1}{q}}, & \text{if } 1\leq q<+\infty, \\
    \dis \sup_{t>0} t^{\frac{1}{p}}v_*(t), & \text{if } q=+\infty.
\end{cases}
\end{equation*}
Here, $v_*$ denotes the decreasing rearrangement of $v$, defined by
\[
v_*(s)=\inf\{t>0: |\{ x\in\Omega: |v(x)|>t\}|\leq s\}.
\]
When $p=q$, the Lorentz space coincides with the standard Lebesgue space, that is, $L^{p,p}(\Omega)=L^{p}(\Omega)$. Furthermore, $L^{p,q_1}(\Omega) \subset L^{p,q_2}(\Omega)$ whenever $q_1<q_2$. Moreover, since $|\Omega|<+\infty$, one also has $L^{p_2,q}(\Omega) \subset L^{p_1,r}(\Omega)$ for $p_1<p_2$ and any $1\leq q,r\leq +\infty$.

Finally, we emphasize that the H\"{o}lder inequality
\[
\into |fg| \leq \|f\|_{L^{p,q}(\Omega)} \|g\|_{L^{p',q'}(\Omega)},
\]
holds for $1<p<+\infty$ and $1\leq q\leq +\infty$, where $p' = \frac{p}{p-1}$ and $q' = \frac{q}{q-1}$. This inequality is a direct consequence of the Hardy--Littlewood inequality. We refer the reader to~\cite{PKOF} for a more in-depth introduction to Lorentz spaces.

\subsection{A generalized Gauss-Green formula}

The $L^\infty$-measure-divergence vector fields theory (also known as Anzellotti's theory) plays a fundamental role in problems involving the 1-Laplacian operator. In essence, it develops an entire theory whose ultimate goal is to obtain a Green's formula (see~\eqref{eq:Green}) for very irregular functions. To be precise, let us begin by defining the set
\[
\DM(\Omega):= \big\{ z\in L^\infty(\Omega)^N : \operatorname{div}z \in \mathcal{M}(\Omega) \big\},
\]
and let us denote by $\DM_{\rm loc}(\Omega)$ its local version, i.e., the set of bounded vector fields $z$ with ${\rm div}z\in \mathcal{M}_{\rm loc} (\Omega)$. It is worth noting that, as proved in~\cite[Proposition~3.1]{ChenFrid}, $\operatorname{div}z$ is absolutely continuous with respect to the measure $\mathcal{H}^{N-1}$ for any vector field $z\in \DM (\Omega)$. Given a vector field $z\in \DM_{\rm loc} (\Omega)$ and a function $u\in BV_{\rm loc}(\Omega)$, we define the distribution $(z,Du) \colon C_c^\infty(\Omega) \to \R$ as
\begin{equation} \label{eq:Pairing}
\langle(z,Du),\varphi\rangle:=-\int_\Omega u^*\varphi\operatorname{div}z-\int_\Omega uz\cdot\nabla\varphi, \,\quad \forall \varphi\in C_c^\infty(\Omega).
\end{equation}
This distribution was introduced by Anzellotti (\cite{Anz}) and it defines a generalized dot product between $z$ and $Du$. Observe that, without further assumptions, the first integral in~\eqref{eq:Pairing}, and hence the one in~\eqref{eq:Green},  is not necessarily finite. Therefore, in order to give meaning to this term, one typically needs to pay an additional price by increasing the regularity of either $z$ or $u$. For instance, one may assume $z\in \DM(\Omega)$ with $\operatorname{div}z\in L^N(\Omega)$ and $u\in \bvo$ (see~\cite{Anz}), or alternatively $z\in \DM(\Omega)$ and $u\in \bvo \cap \lio$ (see~\cite{Ca}).

The most general scenario is studied in~\cite{CrDec} (see also~\cite{DecGiSe}). For $z\in \DM_{\rm loc} (\Omega)$ and $u^*\in BV_{\rm loc}(\Omega) \cap L_{\rm loc}^1(\Omega,  |\operatorname{div}z|)$, the authors showed that the distribution $(z,Du)$ is, in fact, a Radon measure that has, locally, finite total variation. Furthermore, for every Borel set $B$ and every open set $U$ with $B \subset U \subset \Omega$ it holds
\begin{equation}
\label{eq:AbsCont}
    \left| \int_B (z, Du) \right| \leq \int_B |(z, Du)| \leq \|z\|_{L^{\infty}(U)^N} \int_B |Du|.
\end{equation}
The measure $(z,Du)$ actually generalizes the usual dot product. It admits a decomposition into two parts with respect to the $N$-dimensional Lebesgue measure $\mathcal{L}^N$: an absolutely continuous part and a singular one. The absolutely continuous part, denoted by $(z,Du)^a$, satisfies that $(z,Du)^a = z\cdot D^a u = z\cdot \nabla u\, \mathcal{L}^N$ (see~\cite[Theorem~4.12]{CrDec}). Thus, for functions $\varphi \in W^{1,1}(\Omega)$, the pairing $(z,D\varphi)$ just turns out to be $(z,D\varphi) = z\cdot \nabla \varphi\, \mathcal{L}^N$.

In practice, due to the nature of our problem, we always deal with a non-positive $\operatorname{div}z$. In this case, any $u\in BV_{\rm loc}(\Omega)$ satisfies $u^*\in L_{\rm loc}^1(\Omega,  |\operatorname{div}z|)$ provided that $z\in \DM_{\rm loc}(\Omega)$ (see \cite[Remark~4.14]{CrDec}). Then, in our particular setting, the entire theory can be developed under the sole assumptions that $z\in \DM_{\rm loc}(\Omega)$ and $u\in BV_{\rm loc}(\Omega)$. We also emphasize that if stronger regularity holds, namely, if $u\in BV(\Omega)$ and $z\in \DM(\Omega)$ with $\operatorname{div}z \leq 0$, then $u^*\in L^1(\Omega, \operatorname{div}z)$ (\cite[Corollary~3.3]{AbDaSe}). To simplify notation, we will write $u$ instead of $u^*$ whenever no ambiguity can arise; for instance, when $u$ lacks a jump part.

A direct consequence of~\eqref{eq:AbsCont} is that $(z,Du)$ is absolutely continuous with respect to the measure $|Du|$. Denoting by \(\theta(z, Du, \cdot)\)
the Radon-Nikodym derivative of $(z, Du)$ with respect to $|Du|$, it holds
\begin{equation*}
	(z, Du) = \theta(z,Du,x) \, {|Du|} \text{ as measures in } \Omega.
\end{equation*}
It can be deduced that this derivative behaves well with respect to the composition, as the following result shows.

\begin{lemma}
	\label{lem:Composition}
	Let $z \in \mathcal{DM}^\infty_{\rm loc}(\Omega)$ and $u \in BV_{\rm loc}(\Omega) \cap L_{\rm loc}^1(\Omega,  | \operatorname{div}z|)$. Let $\Psi: I \to \mathbb{R}$ be a continuous strictly increasing function, with $I \subseteq \mathbb{R}$ an interval. Assume that $u(\Omega)\subseteq I$ and $\Psi(u) \in BV_{\rm loc}(\Omega) \cap L_{\rm loc}^1(\Omega,  |\operatorname{div}z|)$.
	\begin{enumerate}
		\item[(i)] If $\Psi$ is locally absolutely continuous in $\mathrm{Int}(I)$ then
		\begin{equation}\label{theta_i}
		\theta(z, D\Psi(u),x) = \theta(z, Du,x) \quad \text{for } |D\Psi(u)|\text{-a.e. } x \in \Omega.
		\end{equation}
		Moreover, $|D\Psi(u)| \ll|Du|$ and, as a consequence, $(z,Du)=|Du|$ implies $(z,D\Psi(u))= |D\Psi(u)|$.
		\item[(ii)] If $\Psi'>0$ a.e. in $I$ then
		\begin{equation}\label{theta_ii}
		\theta(z, D\Psi(u),x) = \theta(z, Du,x) \quad \text{for } |Du|\text{-a.e. } x \in \Omega.
		\end{equation}
		Moreover, $|Du|\ll |D\Psi(u)|$ and, as a consequence, $(z,D\Psi(u))= |D\Psi(u)|$ as measures implies $(z,Du)=|Du|$ as measures.
	\end{enumerate}
\end{lemma}

\begin{remark}\label{rem:Composition_nondecr} 
    Lemma \ref{lem:Composition}(i) also holds if one simply assumes that $\Psi$ is non-decreasing instead of increasing. Indeed, for $\varepsilon>0$ we can apply Lemma \ref{lem:Composition}(i) to the function $\Psi_\varepsilon(t):=\Psi(t)+\varepsilon t$ which is continuous, strictly increasing and locally absolutely continuous in $\mathrm{Int}(I)$. First, observe that $|D\Psi(u)|\ll |Du|$ since $|D\Psi(u)| \leq |D\Psi_1(u)|+|Du|$ and $|D\Psi_1(u)| \ll |Du|$ by Lemma~\ref{lem:Composition}(i). On the other hand, by Lemma~\ref{lem:Composition}(i) and the linearity of the pairing, we have
	\[
	(z,D\Psi(u))+\varepsilon(z,Du)=(z,D\Psi_\varepsilon(u))=\theta(z, Du,x)|D\Psi(u)+\varepsilon Du| \quad \text{as measures in } \Omega.
	\]
	Let $B\subset\subset\Omega$ be a Borel set. Since
	\[
	\big|\,|D\Psi(u)+\varepsilon Du|(B)-|D\Psi(u)|(B)\,\big|\le\varepsilon|Du|(B)
	\qquad\text{and}\qquad
	|(z,Du)|(B)\le\|z\|_{L^\infty(\Omega)^N}|Du|(B)
	\]
	(see \eqref{eq:AbsCont}), we deduce that
	\[
	(z,D\Psi(u))(B)=\theta(z, Du,x)|D\Psi(u)+\varepsilon Du|(B)-\varepsilon(z,Du)(B)\xrightarrow[\varepsilon\to0]{} \theta(z, Du,x)|D\Psi(u)|(B).
	\]
	As the left-hand side does not depend on $\varepsilon$, we get $(z,D\Psi(u))(B)=\theta(z, Du,x)|D\Psi(u)|(B)$ for every Borel set $B\subset\subset\Omega$, that is, $\theta(z, D\Psi(u),x) = \theta(z, Du,x)$ for $|D\Psi(u)|$-a.e. $x \in \Omega$.
\end{remark}

\begin{remark}\label{rem:Composition_nonincr}
	Also note that if  $\Psi$ is continuous, non-increasing and locally absolutely continuous in $\mathrm{Int}(I)$, applying the above Remark \ref{rem:Composition_nondecr} to $-\Psi$ and using $(z,D(-v))=-(z,Dv)$, $|D(-v)|=|Dv|$, we obtain that $(z,Du)=|Du|$ implies $(z,D\Psi(u))= -|D\Psi(u)|$.
	Analogously, if $\Psi'<0$ a.e.\ in $I$, an application of Lemma \ref{lem:Composition} (ii) gives that $(z,D\Psi(u))= -|D\Psi(u)|$ implies $(z,Du)=|Du|$.
\end{remark}

\begin{remark}\label{rem:Composition_suff}
	As a consequence of the chain rule, we point out some sufficient conditions for the assumptions of Lemma~\ref{lem:Composition}.
	\begin{enumerate}
		\item[(a)] If $u\in BV_{\rm loc}(\Omega)$, $\Psi$ is locally Lipschitz in $I$ and for every $\omega\subset\subset\Omega$ there exists a compact set $K_\omega\subset I$ such that $u\in K_\omega$ a.e. in $\omega$, then $\Psi$ is locally absolutely continuous in ${\rm Int}(I)$ and $\Psi(u)\in BV_{\rm loc}(\Omega)\cap L^\infty_{\rm loc}(\Omega)$; in particular, $u,\Psi(u)\in L^1_{\rm loc}(\Omega,\operatorname{div}z)$.
		\item[(b)] If $\Psi$ is bounded and Lipschitz in $I$, the same conclusion holds for every $u \in BV_{\rm loc}(\Omega)$, without assuming $u$ locally bounded.
		\item[(c)] If $I=[a,b]$ and $\Psi \in C([a,b]) \cap C^1((a,b])$ is non-decreasing, then $\Psi$ is locally absolutely continuous in $(a,b)$, regardless of the behaviour of $\Psi'$ near $a$. If moreover $\Psi'>0$ in $(a,b)$, both (i) and (ii) of Lemma~\ref{lem:Composition} apply once $\Psi(u) \in BV_{\rm loc}(\Omega)$, so that $|Du|$ and $|D\Psi(u)|$ are mutually absolutely continuous and
		\[
		(z,Du)=|Du| \quad \Longleftrightarrow \quad (z,D\Psi(u))= |D\Psi(u)|.
		\]
	\end{enumerate}
\end{remark}

\begin{remark}\label{rem:Composition_sharp}
	The assumptions in (i) and (ii) of Lemma~\ref{lem:Composition} are of different nature and neither implies the other. Continuity and strict monotonicity of $\Psi$, even together with $\Psi \in C^1(I)$, are not enough for the conclusion of (ii): if $\Psi'$ vanishes on a set of positive measure with empty interior (for instance, $\Psi(t)=\int_0^t \operatorname{dist}(r,Z)\,dr$ with $Z$ a fat Cantor set), then $|Du|$ may charge a set which is $|D\Psi(u)|$-negligible. Similarly, $\Psi'>0$ a.e.\ without absolute continuity is not enough for (i). In particular, under the assumptions of~\cite[Proposition~2.8]{Anz} (that is, $\Psi$ increasing and of class $C^1$), the argument yields the identity of (i), whereas the identity $|Du|$-a.e.\ requires in addition that $\Psi'>0$ a.e.
\end{remark}

\begin{proof}[Proof of Lemma~\ref{lem:Composition}]
	The proof relies on a classical argument from \cite{Anz}. We set
	\[
	G=\big\{ x\in \Omega: \theta(z, D\Psi(u),x) \neq \theta(z, Du,x)\big\},
	\]
	and we denote, for any $s, t \in \mathbb{R}$, the level sets of $u$ and $\Psi(u)$ by
	\[
	E_{u,t} := \{x \in \Omega : u(x) > t\} \qquad \text{ and } \qquad E_{\Psi(u), s} := \{x \in \Omega : \Psi(u(x)) > s\}.
	\]
	Taking into account that $u, \Psi(u) \in BV_{\rm loc}(\Omega) \cap L_{\rm loc}^1(\Omega, |\operatorname{div}z|)$, an application of~\cite[Theorem~4.2]{CrDec} leads to the existence of two sets $N_u,\ N_{\Psi (u)}\subset \mathbb{R}$ with $\mathcal{L}^1(N_u)=\mathcal{L}^1(N_{\Psi(u)})=0$ such that
	\begin{align}
		&\text{for any } t\in \mathbb{R}\setminus N_{u}, \text{ it holds }\theta(z, Du, x) = \theta(z, D\chi_{E_{u,t}}, x) \quad \text{for } |D\chi_{E_{u,t}}|\text{-a.e. } x \in \Omega, \label{eq:Pf_L_Comp_2} \\
		&\text{for any } s\in \mathbb{R}\setminus N_{\Psi(u)}, \text{ it holds } \theta(z, D\Psi(u), x) = \theta(z, D\chi_{E_{\Psi(u), s}}, x) \quad \text{for } |D\chi_{E_{\Psi(u), s}}| \text{-a.e. } x \in \Omega. \label{eq:Pf_L_Comp_3}
	\end{align}
	Now observe that as $\Psi$ is continuous and strictly increasing then $E_{u,t}$ and $E_{\Psi(u),\Psi(t)}$ are identical for every $t\in I$ and, equivalently, $E_{u,\Psi^{-1}(s)}$ and $E_{\Psi(u),s}$ are identical for every $s\in\Psi(I)$. Therefore
	\begin{equation} \label{eq:Pf_L_Comp_1}
		D\chi_{E_{u,t}} = D\chi_{E_{\Psi(u), \Psi(t)}} \quad \text{for every } t\in I, \qquad\text{i.e.}\qquad D\chi_{E_{u,\Psi^{-1}(s)}} = D\chi_{E_{\Psi(u), s}} \quad \text{for every } s\in\Psi(I).
	\end{equation}
	On the other hand, since $u(\Omega)\subseteq I$, for every $t\in\mathbb{R}\setminus I$ the set $E_{u,t}$ is either $\Omega$ or $\emptyset$, and thus $D\chi_{E_{u,t}}=0$. Similarly, $D\chi_{E_{\Psi(u), s}}=0$ for every $s\in\mathbb{R}\setminus\Psi(I)$.
	
	\medskip
	\noindent\emph{Proof of (i).} 	Gathering together \eqref{eq:Pf_L_Comp_2}, \eqref{eq:Pf_L_Comp_3}, and \eqref{eq:Pf_L_Comp_1}, for every $s\in \mathbb{R}\setminus \left(\Psi(N_u)\cup N_{\Psi(u)} \right)$ we have that
	\begin{equation}
		\label{eq:Pf_L_Comp_4}
		\theta(z, Du, x) = \theta(z, D\chi_{E_{u,\Psi^{-1}(s)}}, x) = \theta(z, D\chi_{E_{\Psi(u),s}},x) = \theta(z, D\Psi(u), x),\quad \text{for } |D\chi_{E_{\Psi(u),s}}|\text{-a.e. } x\in\Omega.
	\end{equation}
	Note that the set $\Psi(N_u)$ has zero $\mathcal{L}^1$-measure, since $\Psi$ is locally absolutely continuous in $I$ and therefore satisfies the Lusin $(N)$ property (see, e.g., \cite[Corollary~3.41]{leoni}). Thus,~\eqref{eq:Pf_L_Comp_4} implies that $|D\chi_{E_{\Psi(u),s}}|(G)=0$ for almost every $s\in\mathbb{R}$ and, thanks to the coarea formula, we conclude that $|D\Psi(u)|(G)=0$. This shows the validity of \eqref{theta_i}.

	Finally, we show that $|D\Psi(u)| \ll |Du|$. Suppose that $B\subset \Omega$ is such that $|Du|(B)=0$. By the coarea formula, $|D\chi_{E_{u,t}}|(B) = 0$ for every $t\in \R\setminus \tilde N_u$ with $\mathcal{L}^1(\tilde N_u)=0$. Using~\eqref{eq:Pf_L_Comp_1}, it follows that $|D\chi_{E_{\Psi(u),s}}|(B) = 0$ for every $s\in \R\setminus \Psi(\tilde N_u)$. Since $\Psi(\tilde N_u)$ has zero $\mathcal{L}^1$-measure because $\Psi$ satisfies the Lusin $(N)$ property, the coarea formula yields $|D\Psi(u)|(B)=0$.
   
	\medskip
	\noindent\emph{Proof of (ii).} In this case, once again gathering together \eqref{eq:Pf_L_Comp_2}, \eqref{eq:Pf_L_Comp_3}, and \eqref{eq:Pf_L_Comp_1}, for every $t\in I\setminus \left(N_u\cup\Psi^{-1}(N_{\Psi(u)}) \right)$ we have that
	\begin{equation}
		\label{eq:Pf_L_Comp_5}
		\theta(z, Du, x) = \theta(z, D\chi_{E_{u,t}}, x) = \theta(z, D\chi_{E_{\Psi(u),\Psi(t)}},x) = \theta(z, D\Psi(u), x),\quad \text{for } |D\chi_{E_{u,t}}|\text{-a.e. } x\in\Omega.
	\end{equation}
	Now, as $\Psi'>0$ a.e. in $I$, one gains that $\Psi^{-1}(N_{\Psi(u)})$ has zero $\mathcal{L}^1$-measure. Recalling that $D\chi_{E_{u,t}}=0$ for $t\in\mathbb{R}\setminus I$,~\eqref{eq:Pf_L_Comp_5} implies that $|D\chi_{E_{u,t}}|(G)=0$ for almost every $t\in\mathbb{R}$ and, thanks to the coarea formula, we conclude that $|Du|(G)=0$. This shows \eqref{theta_ii}.

	Now we show that $|Du| \ll |D\Psi(u)|$. Assume that $B\subset \Omega$ is such that $|D\Psi(u)|(B)=0$. By the coarea formula, $|D\chi_{E_{\Psi(u),s}}|(B) = 0$ for every $s\in \R\setminus \tilde N_{\Psi(u)}$, where $\mathcal{L}^1(\tilde N_{\Psi(u)})=0$. Taking~\eqref{eq:Pf_L_Comp_1} into account, it follows that $|D\chi_{E_{u,t}}|(B) = 0$ for every $t\in \R\setminus \Psi^{-1}(\tilde N_{\Psi(u)})$. Note that $\Psi^{-1}(\tilde N_{\Psi(u)})$ has zero $\mathcal{L}^1$-measure because $\Psi'>0$ a.e. in I. Thus, the coarea formula implies $|Du|(B)=0$.
\end{proof}

A general Leibniz rule for the pairing also holds. In~\cite[Proposition~4.11]{CrDec}, it is shown that if $z\in \DM_{\rm loc}(\Omega)$ and $u,v\in BV_{\rm loc}(\Omega)\cap L^\infty_{\rm loc} (\Omega)$, then
\begin{equation}
\label{eq:Prod}
    (z, D(uv)) = u^*(z,Dv) + v^*(z,Du).
\end{equation}

Since $\Omega$ has a Lipschitz boundary, the outward normal unit vector $\nu(x)$ is defined for $\mathcal H^{N-1}$-almost every $x\in\partial\Omega$. In~\cite{Anz}, it is shown that every $z \in \mathcal{DM}^{\infty}(\Omega)$ possesses
a weak trace on $\partial \Omega$ of the
normal component of $z$ which is denoted by
$[z, \nu]$. This notion of weak trace generalizes the classical one, in the sense that $[z, \nu] = z\cdot \nu$ on $\partial\Omega$ if $z\in C^1(\overline{\Omega}, \R^N)$. Moreover, it verifies
\begin{equation*}
\| [z,\nu] \|_{L^\infty (\partial\Omega)}\leq \|z\|_{L^\infty(\Omega)^N}.
\end{equation*}

A Green-type formula involving the pairing $(z,Du)$ and the weak trace $[z,\nu]$ is stated as follows. 

\begin{proposition}{\cite[Proposition~2.5]{DecGiSe}}
Let $z\in \DM(\Omega)$ and $u\in BV(\Omega) \cap L^1 (\Omega, \operatorname{div}z)$. Then, it holds
\begin{equation} \label{eq:Green}
	\int_{\Omega} u^* \operatorname{div}z + \int_{\Omega} (z, Du) =
	\int_{\partial \Omega} u[z, \nu] \ \dH.
\end{equation}
\end{proposition}

Finally, as we have said, in practice we always deal with a non-positive $\operatorname{div}z$. In this case, one has the following regularity result.
\begin{lemma}{\cite[Lemma~2.3]{GOP}}
\label{lem:Global_z}
Let $z\in L^{\infty}(\Omega)^N$ be such that
\[
-\operatorname{div}z \geq 0 \text{ in } \mathcal{D}'(\Omega).
\]
Then $\operatorname{div}z \in \mathcal{M}(\Omega)$ and $z\in \DM(\Omega)$.  
\end{lemma}

\subsection{On the jump part of the solutions}

To prove that solutions to~\eqref{eq:PbIntro1} have no jump part, we follow a strategy introduced in~\cite{BalOP} that involves the definition of a pairing different (though similar) from~\eqref{eq:Pairing}. Let $\beta\colon\mathbb R\to\mathbb R$ be a continuous function. Following the idea presented in \cite{MaSe} we define the distribution $\big(z,D\beta(v)^\#\big)\colon C^\infty_c(\Omega)\to\mathbb R$ as
\begin{equation}\label{eq:new_Pairing}
	\big\langle\big(z,D\beta(v)^\#\big),\varphi\big\rangle:=-\int_\Omega\beta(v)^\#\varphi\,\operatorname{div}z-\int_\Omega\beta(v)\,z\cdot\nabla\varphi,\qquad\forall\varphi\in C^\infty_c(\Omega),
\end{equation}
where $\beta(v)^\#$ is defined in \eqref{eq:def_beta_hash}. Let us prove the following proposition which shows that the previous pairing defines a finite Radon measure.
\begin{proposition}\label{prop:sharp_pairing}
	Let $z\in\mathcal{DM}^\infty_{\rm loc}(\Omega)$, $v\in BV_{\rm loc}(\Omega)$ and let $\beta\colon\mathbb R\to\mathbb R$ be continuous with $\beta(v)\in BV_{\rm loc}(\Omega)\cap L^\infty_{\rm loc}(\Omega)$ and $\beta(v)^\#\in L^\infty_{\rm loc}(\Omega,|\operatorname{div}z|)$. Then $\big(z,D\beta(v)^\#\big)$ is a Radon measure in $\Omega$ and
	\begin{equation}\label{eq:sharp_decomp}
    \big(z,D\beta(v)^\#\big)=\big(z,D\beta(v)\big)+\Big(\frac{\beta(v^+)+\beta(v^-)}{2}-\beta(v)^\#\Big)\operatorname{div}z \llcorner J_v\qquad\text{as measures in }\Omega.
	\end{equation}
	If, in addition, $\beta$ is monotone, then the assumption $\beta(v)^\#\in L^\infty_{\rm loc}(\Omega,|\operatorname{div}z|)$ is automatically satisfied and, for every Borel set $B$ and every open set $U$ with $B\subset U\subset\Omega$,
	\begin{equation}\label{eq:sharp_estimate}
		\Big|\int_B\big(z,D\beta(v)^\#\big)\Big|\le\int_B\big|\big(z,D\beta(v)^\#\big)\big|\le\|z\|_{L^\infty(U)^N}\int_B|D\beta(v)| .
	\end{equation}
\end{proposition}

\begin{proof}
	First observe that, for every $\varphi\in C^\infty_c(\Omega)$, subtracting \eqref{eq:Pairing} (with $u=\beta(v)$) from \eqref{eq:new_Pairing} gives
	\[
	\big\langle\big(z,D\beta(v)^\#\big),\varphi\big\rangle=\big\langle\big(z,D\beta(v)\big),\varphi\big\rangle+\int_\Omega\big(\beta(v)^*-\beta(v)^\#\big)\varphi\,\operatorname{div}z .
	\]
	The first term on the right is a Radon measure satisfying \eqref{eq:AbsCont}. The second is the measure $(\beta(v)^*-\beta(v)^\#)\operatorname{div}z$, since $\beta(v)^*-\beta(v)^\#$ is a Borel function which is locally bounded $|\operatorname{div}z|$-a.e. Hence $(z,D\beta(v)^\#)$ is a Radon measure. Since $\mathcal H^{N-1}(S_v\setminus J_v)=0$ and $|\operatorname{div}z|$ vanishes on $\mathcal H^{N-1}$-negligible sets, the expressions for $\beta(v)^*$ above give \eqref{eq:sharp_decomp}.
	
	Let now $\beta$ be monotone. By  \eqref{eq:sharp_decomp}, $(z,D\beta(v)^\#)\llcorner(\Omega\setminus J_v)=(z,D\beta(v))\llcorner(\Omega\setminus J_v)$. Hence \eqref{eq:AbsCont} gives \eqref{eq:sharp_estimate} for $B\setminus J_v$. It remains to show that \eqref{eq:sharp_estimate} holds for $B\subseteq J_v$
	
	Since $J_v$ is countably $\mathcal H^{N-1}$-rectifiable \cite[Theorem 3.78]{AFP}, then there exist $C^1$ hypersurfaces $\xi_i$ such that $\mathcal{H}^{N-1} \left( S_{v} \setminus \bigcup_{i=1}^\infty \xi_i \right) =0$. These hypersurfaces are oriented, $\mathcal H^{N-1}$-a.e.\ on $\xi_i\cap J_v$, by $\nu=\nu_v$. By countable additivity it suffices to prove \eqref{eq:sharp_estimate} for Borel sets $B\subset\xi_i\cap J_v\cap U$. The traces of $\beta(v)$ on the two sides of $\xi_i$ are $\beta(v^\pm)$ $\mathcal H^{N-1}$-a.e.\ on $\xi_i\cap J_v$. By \eqref{eq:new_Pairing}, $(z,D\beta(v)^\#)=\operatorname{div}(\beta(v)z)-\beta(v)^\#\operatorname{div}z$; then applying \cite[Proposition 3.4]{AmCrMa} to $z$ and to $\beta(v)z$, together with \cite[Corollary 3.2]{CrDec}, we obtain
	\begin{equation}\label{eq:sharp_density}
		\big(z,D\beta(v)^\#\big) \llcorner(\xi_i\cap J_v)=\Big([z,\nu]^+\big(\beta(v^+)-\beta(v)^\#\big)+[z,\nu]^-\big(\beta(v)^\#-\beta(v^-)\big)\Big)\mathcal H^{N-1}\llcorner(\xi_i\cap J_v),
	\end{equation}
	with $|[z,\nu]^\pm|\le\|z\|_{L^\infty(U)^N}$ on $\xi_i\cap U$. Since, by monotonicity, $\beta(v)^\#$ lies between $\beta(v^-)$ and $\beta(v^+)$, we have $|\beta(v^+)-\beta(v)^\#|+|\beta(v)^\#-\beta(v^-)|=|\beta(v^+)-\beta(v^-)|$. Thus
	\[
	\big|\big(z,D\beta(v)^\#\big)\big|(B)\le\|z\|_{L^\infty(U)^N}\int_B|\beta(v^+)-\beta(v^-)|\,d\mathcal H^{N-1}=\|z\|_{L^\infty(U)^N}|D\beta(v)|(B).
	\]
	The last equality holds because $|D\beta(v)|\llcorner J_v=|\beta(v^+)-\beta(v^-)|\,\mathcal H^{N-1}\llcorner J_v$. Indeed, $J_{\beta(v)}\subset J_v$ up to an $\mathcal H^{N-1}$-negligible set, and $|D\beta(v)|$ vanishes on $J_v\setminus J_{\beta(v)}$, where $\beta(v^+)=\beta(v^-)$.
\end{proof}

\begin{remark}\label{rem:sharp_nonmonotone}
	The monotonicity of $\beta$ cannot be dropped in \eqref{eq:sharp_estimate}, even for smooth $\beta$ and bounded $v$. Let
	\[
	\Omega=(-1,1)^N,\quad\Sigma=\Omega\cap\{x_N=0\},\quad v=\chi_{\{x_N>0\}},\quad\beta(s)=s(1-s),\quad z(x)=-\operatorname{sign}(x_N)\,e_N .
	\]
	Then $v\in BV(\Omega)\cap L^\infty(\Omega)$ with $J_v=\Sigma$, $v^+=1$, $v^-=0$ and $\nu_v=e_N$. Since $\beta(0)=\beta(1)=0$, we have $\beta(v)\equiv0$, and so $|D\beta(v)|=0$. An integration by parts gives $\operatorname{div}z=-2\,\mathcal H^{N-1}\llcorner\Sigma$, so $z\in\mathcal{DM}^\infty(\Omega)$. On $\Sigma$, $\beta(v)^\#=\int_0^1s(1-s)\,ds=\frac16$. Hence \eqref{eq:new_Pairing} yields $(z,D\beta(v)^\#)=\frac13\,\mathcal H^{N-1}\llcorner\Sigma\neq0$, contradicting \eqref{eq:sharp_estimate}. In terms of \eqref{eq:sharp_density}, $[z,\nu]^+=-1$ and $[z,\nu]^-=1$, while $\beta(v)^\#=\frac16$ does not lie between $\beta(v^-)=\beta(v^+)=0$. 
\end{remark}
The next result states that functions $v\in BV_{\rm loc}(\Omega)$ verifying a distributional inequality involving its gradient have no jump part.

\begin{lemma}
\label{lem:no_jump}
Let $\alpha, \beta \colon \R\to \R$ be two continuous increasing functions. Let $z\in \DM_{\rm loc} (\Omega)$ be such that $\|z\|_{\lio^N}\leq 1$, and let $v$ satisfy $v, \alpha(v)\in BV_{\rm loc}(\Omega)$ and $\beta(v)\in BV_{\rm loc}(\Omega) \cap L^\infty_{\rm loc}(\Omega)$. Given $f \in L^1_{\rm loc}(\Omega)$, assume that
\begin{equation}
\label{eq:Pf_No_jump_H1}
-\operatorname{div}z \geq |D\alpha(v)| + f
\end{equation}
and
\begin{equation}
\label{eq:Pf_No_jump_H2}
\left(z, D\beta(v)^\# \right) = |D\beta(v)|.
\end{equation}
Then $D^j v =0$.
\end{lemma}

\begin{proof}
We follow the same ideas as in~\cite[Lemma~2.3]{BalOP}. Since $v \in BV_{\rm loc}(\Omega)$, by~\cite[Theorem~3.78]{AFP}, the jump set $J_v$ is a countably $\mathcal{H}^{N-1}$-rectifiable set. Then, there exist $C^1$ hypersurfaces $\xi_i$, oriented $\mathcal H^{N-1}$-a.e.\ on $\xi_i\cap J_v$ by $\nu=\nu_v$, such that $\mathcal{H}^{N-1} \left( S_{v} \setminus \bigcup_{i=1}^\infty \xi_i \right) =0$. By countable additivity, the proof follows once we prove $Dv \llcorner \xi_i=0$ for any $i\in \N$. Observe also that, by strict monotonicity, $J_{\alpha(v)} = J_{\beta(v)} = J_v$.

First of all, note that, by~\cite[Proposition~3.4]{AmCrMa}, one has that
\begin{equation}
\label{eq:Pf_No_jump_1}
\operatorname{div} z \llcorner \xi_i = ([z,\nu]^+ - [z,\nu]^-) \mathcal H^{N-1}\llcorner \xi_i,
\end{equation}
where $[z,\nu]^+$ and $[z,\nu]^-$ are the traces of the normal component of $z$ over $\xi_i$ defined as in~\cite[Definition~3.3]{AmCrMa} (see also~\cite[Section~2.4]{CrDec}). By \eqref{eq:new_Pairing}, $(z,D\beta(v)^\#)=\operatorname{div}(\beta(v)z)-\beta(v)^\#\operatorname{div}z$. Since $\beta(v)\in BV_{\rm loc}(\Omega) \cap L^\infty_{\rm loc}(\Omega)$, we can use~\cite[Proposition~3.1]{CrDec} and~\eqref{eq:Pf_No_jump_1} to deduce that
\begin{equation}
\label{eq:Pf_No_jump_2}
\left(z, D\beta(v)^\# \right)\llcorner \xi_i = \Big( [z,\nu]^+ \left( \beta(v)^+ - \beta(v)^\# \right) + [z,\nu]^- \left( \beta(v)^\# - \beta(v)^- \right) \Big) \mathcal H^{N-1}\llcorner \xi_i.
\end{equation}
Observe that $\beta(v)^\# \in \left[ \min\{\beta(v)^-, \beta(v)^+\}, \max\{\beta(v)^-, \beta(v)^+\} \right]$ as $\beta$ is increasing. Distinguishing between cases $\beta(v)^+ - \beta(v)^- \geq 0$ and $\beta(v)^+ - \beta(v)^- < 0$, one easily obtains from~\eqref{eq:Pf_No_jump_2} that
\begin{equation}
\label{eq:Pf_No_jump_3}
\left(z, D\beta(v)^\# \right) \llcorner \xi_i \leq |\beta(v)^+ - \beta(v)^-|\mathcal H^{N-1}\llcorner \xi_i.
\end{equation}
with the equality sign if and only if $[z,\nu]^+ = [z,\nu]^- = \operatorname{sign}(\beta(v)^+ - \beta(v)^-)$.

Now, since
\[
|D\beta(v)| \llcorner \xi_i = |\beta(v)^+ - \beta(v)^-| \mathcal H^{N-1}\llcorner \xi_i
\]
and we are assuming~\eqref{eq:Pf_No_jump_H2}, we deduce that the equality holds in~\eqref{eq:Pf_No_jump_3}. Therefore, we have
\[
[z,\nu]^+ = [z,\nu]^- \text{ on } \xi_i,
\]
and then, by~\eqref{eq:Pf_No_jump_1}, $\operatorname{div}z = 0$ on $\xi_i$. Consequently, from~\eqref{eq:Pf_No_jump_H1} we obtain that  
\[
|D\alpha(v)|\llcorner \xi_i = |\alpha(v)^+ - \alpha(v)^-| \mathcal H^{N-1}\llcorner \xi_i = 0 
\]
Since $\alpha$ is increasing, then $D v\llcorner \xi_i = 0$. We conclude that $D^j v=0$.
\end{proof}

\section{Setting and main results}
\label{sec:Main}

In this section, we present several  results concerning   existence and   qualitative properties of the solutions to the Dirichlet problem
\begin{equation}
	\label{eq:PbMain}
	\begin{cases}
		-\Delta_1 u = g(u)|Du| + h(u)f & \text{in }\Omega,\\
		u=0 & \text{on } \partial\Omega,
	\end{cases}
\end{equation}
where $0<f\in L^{N,\infty}(\Omega)$. We assume that $g\colon [0,+\infty) \to (0,+\infty]$ and $h\colon [0,+\infty) \to [0,+\infty]$ are continuous functions, finite outside the origin, which may be singular at zero. In what follows, we denote by $\Gamma\colon [0,+\infty) \to [-\infty,+\infty)$ the increasing function
\[
\Gamma(s) = \int_1^s g(t) \, \mathrm{d}t,\ \forall s\geq 0.
\]
Note that, depending on the integrability of $g$ at the origin, $\Gamma(0)$ may be either finite or $-\infty$.  
  The definition of solution for problem~\eqref{eq:PbMain} is given, for further purposes, in the general case of  a positive datum $f\in L^1(\Omega)$.

\begin{defin}\label{defsol}
Let $0<f\in L^1(\Omega)$. A non-negative function $u\in BV(\Omega)$ is a \textit{solution} to problem~\eqref{eq:PbMain} if $D^j u=0$, $\Gamma(u)\in BV_{\rm loc}(\Omega)$, $h(u)f \in L^1_{}(\Omega)$, and if there exists $z\in \mathcal{D}\mathcal{M}^\infty(\Omega)$ with $\|z\|_{L^\infty(\Omega)^N}\le 1$ such that
\begin{gather} 
    \label{eq:def_dist}
    -\operatorname{div}z = g(u)|Du| + h(u)f \text{ as measures in } \Omega,\\
    \label{eq:def_pair}
    (z,Du)=|Du| \text{ as measures in } \Omega,\\
    \label{eq:def_bord}
    u(x)(1+[z,\nu](x))=0 \text{ for  $\mathcal{H}^{N-1}$-a.e. } x \in \partial\Omega.		
\end{gather}
\end{defin}
\begin{remark}
In the above definition, the vector field $z$ represents the formal quotient $\frac{Du}{|Du|}$ that appears in the 1-Laplacian operator. First introduced in the pioneering paper~\cite{AnBaCaMa}, this formulation has become standard in the literature. In~\eqref{eq:def_dist}, $z$ \textit{replaces} $\frac{Du}{|Du|}$ within the divergence operator, while \eqref{eq:def_pair} \textit{identifies} $z$ as the quotient $\frac{Du}{|Du|}$. Condition~\eqref{eq:def_bord} is how Dirichlet boundary conditions are understood in 1-Laplace-type problems. It appears naturally in minimization problems involving the total variation functional (\cite{AnBaCaMa, ACM, Dem}). A consequence of this very weak boundary condition is that solutions are not forced to be identically zero on $\partial \Omega$. 
\end{remark}

\begin{remark} \label{rem:g|Du|}
Since $D^ju=0$ and $\Gamma(u)\in BV_{\rm loc} (\Omega)$, Lemma~\ref{lem:ChainRule} implies that $g(u)|Du|$ is a finite Radon measure and that the identity $\chi_{\{u>0\}}|D\Gamma(u)| = g(u)|Du|$ holds. 
\end{remark}

\begin{remark}
Any solution in the sense of Definition~\ref{defsol} satisfies the weak formulation
\begin{equation}
\label{eq:weak_form}
\into (z, D\phi) - \intdo \phi [z,\nu] \ \dH = \into g(u) \phi |Du| + \into h(u) \phi f, \ \forall \phi \in \bvo\cap\lio.
\end{equation}
Indeed, since $g,h\geq 0$ and $f>0$, it follows from~\eqref{eq:def_dist} that $\operatorname{div}z\leq 0$ and consequently, as $z\in \DM(\Omega)$, one has $\phi^*\in L^1(\Omega, \operatorname{div}z)$ for every $\phi\in\bvo\cap\lio$ (see~\cite[Corollary~3.3]{AbDaSe}). Hence, taking $\phi$ as test function in~\eqref{eq:def_dist} (note that both $g(u)|Du|$ and $h(u)f$ are finite measures) and applying Green's identity~\eqref{eq:Green}, we obtain~\eqref{eq:weak_form}. 
 
\end{remark}

\begin{remark} 
If $h$ is singular at zero (i.e. $h(0)=+\infty$), then any solution $u$ to~\eqref{eq:PbMain} is strictly positive in $\Omega$; indeed, since $h(u)f \in L^1(\Omega)$ and $f>0$, the set $\{u=0\}$ must have zero measure. Conversely, if $h(0)<+\infty$, problem~\eqref{eq:PbMain} may exhibit a degenerate behaviour, allowing for the existence of trivial solutions even for non-trivial datum $f$ (see Section~\ref{sec:trivial_solutions}).
\end{remark}

We first address the existence of solutions to problem~\eqref{eq:PbMain}, dividing our study into two cases depending on whether $g$ is integrable at the origin or not. In either case, we work under the fundamental assumption that $h$ is bounded at infinity, that is,
\begin{equation}
    \label{eq:hyp_h_inf}
    h(\infty) := \limsup_{s\to +\infty} h(s) < +\infty.
\end{equation}
In case the following integrability condition holds
\begin{equation}\label{eq:hyp_g_int}
\lim_{s\to 0^+} \int_1^s g(t)\, \mathrm{d}t > -\infty,
\end{equation}
we establish the existence of a bounded solution without requiring any further growth conditions on the nonlinearities, provided that the norm $\|f\|_{L^{N,\infty}(\Omega)}$ is sufficiently small. This framework generalizes the setting of~\cite{GOP}, where a similar result is shown under the assumption that $g(s)$ is dominated near the origin by a multiple of $s^{-\theta}$ for $\theta\in(0,1)$, and $h(s)$ is dominated by a multiple of $s^{-\gamma}$ for $\gamma>0$. In contrast, the current approach includes nonlinearities such as  $g(s)=\frac{1}{s(1+\ln^2(s))}$ or $h(s)=e^{\frac{1}{s}}$, which fall outside the framework of~\cite{GOP}.

If the integrability condition~\eqref{eq:hyp_g_int} does not hold, the situation becomes more delicate. In this case, $\Gamma$ is unbounded at the origin and additional assumptions are required to ensure that $\Gamma(u)\in L^1_{\rm loc}(\Omega)$. To this end, we suppose that the essential infimum of $f$ in any open set $\omega$ compactly contained in $\Omega$ is strictly positive, that is, we assume that
\begin{equation}
    \label{eq:hyp_f_pos}
    \forall \omega \subset\subset \Omega \ \ \exists c_\omega>0: f\geq c_\omega \ \ae \text{ in } \omega.
\end{equation}
Moreover, we assume the following power-type growth control for both $g$ and $h$: there exists $\theta>1$ and there exist $\gamma, s_0 > 0$ such that
\begin{equation}
    \label{eq:hyp_sing}
    g(s) \leq \frac{1}{s^\theta} \quad \text{and} \quad h(s) \geq \frac{1}{s^\gamma}, \quad \forall s\in(0,s_0).
\end{equation}
Under these assumptions, we show the existence of a bounded solution to~\eqref{eq:PbMain} even when the integrability condition~\eqref{eq:hyp_g_int} fails, provided that the norm $\|f\|_{L^{N,\infty}(\Omega)}$ is small enough. As already mentioned in the introduction, this result reveals a continuity-breaking phenomenon in relation to the non-existence results known for the $p$-Laplacian. 

In what follows, we present our main existence theorem in a unified manner, covering both the cases where $g$ is integrable at the origin and where it is not.

\begin{theorem}
\label{th:Exist}
Let $0<f\in L^{N,\infty}(\Omega)$. Assume that $h$ satisfies~\eqref{eq:hyp_h_inf} and that $\|f\|_{L^{N,\infty}(\Omega)} < \left(\tilde{\mathcal{S}}_1 h(\infty)\right)^{-1}$. Then problem~\eqref{eq:PbMain} has a solution $u\in\bvo\cap\lio$ provided that one of the following conditions holds:
\begin{enumerate}[i)]
    \item $g$ satisfies the integrability condition~\eqref{eq:hyp_g_int}.
    \item f satisfies~\eqref{eq:hyp_f_pos} and~\eqref{eq:hyp_sing} holds.
\end{enumerate}
\end{theorem}

To prove this theorem, we introduce a novel approach designed to handle singular nonlinearities with uncontrolled growth at the origin. A key element of our analysis is the development of a bootstrap argument, with the use of both~\eqref{eq:hyp_f_pos} and~\eqref{eq:hyp_sing}, which ensures the local integrability of $|\Gamma(u)|$. Our strategy consists of approximating~\eqref{eq:PbMain} by removing the singularities and then passing to the limit. Full details are provided in Section~\ref{sec:Exist}.

Beyond existence, we investigate how the integrability of $g$ at the origin determines the behaviour of the solutions on the boundary. Given that the Dirichlet boundary condition~\eqref{eq:def_bord} is satisfied in a very weak sense, solutions with non-zero boundary trace may arise. Indeed, the next general result shows that the lack of integrability of $g$ at zero prevents the solutions from having a zero trace.

\begin{theorem}
\label{th:Trace}
Let $0<f\in L^1(\Omega)$, and assume that $g$ does not satisfy~\eqref{eq:hyp_g_int} and $h$ is positive.  Then there are no solutions $u\in\bvo$ to~\eqref{eq:PbMain} such that $u\lvert_{\partial\Omega}=0$.  \end{theorem}

The non-integrability of $g$ is an essential requirement for Theorem~\ref{th:Trace}. Indeed, in Section~\ref{sec:Sharpness} we show that problem~\eqref{eq:PbMain} can admit solutions with zero boundary trace provided that $g$ satisfies the integrability condition~\eqref{eq:hyp_g_int}. The proof of Theorem~\ref{th:Trace} is detailed in Section~\ref{sec:Trace}. 

We also note that Theorem \ref{th:Trace} is  in agreement with the phenomenon highlighted in~\cite[Section~6]{GOP}. In~\cite{GOP}, it is shown that there are no solutions with zero trace when $g(s)=\lambda s^{-1}$ with $\lambda\ge 1$, a function that does not satisfy the integrability condition~\eqref{eq:hyp_g_int}. Theorem~\ref{th:Trace} above shows that this phenomenon actually occurs whenever $g$ fails to satisfy~\eqref{eq:hyp_g_int}.

Finally, we point out that all the results stated in this section remain valid provided $g>0$ almost everywhere in $[0,+\infty)$.

\section{Existence of bounded solutions}
\label{sec:Exist}

This section is devoted to the proof of Theorem~\ref{th:Exist}. First, we introduce the approximate problems and obtain uniform estimates for the sequence of solutions. In particular, in the framework of Theorem~\ref{th:Exist}\,\textit{ii)}, we employ a bootstrap argument to ensure that $\Gamma(u)\in L^1_{\rm loc}(\Omega)$. Subsequently, by means of a suitable (and new) technique that is designed to avoid singularities, we pass to the limit to obtain a solution to~\eqref{eq:PbMain}.

From now on, we assume that $0<f\in L^{N,\infty}(\Omega)$, that $h$ satisfies~\eqref{eq:hyp_h_inf} and that $\|f\|_{L^{N,\infty}(\Omega)} < \left(\tilde{\mathcal{S}}_1 h(\infty)\right)^{-1}$. We define, for every $s \in \mathbb{R}$ and $n\in\N$, the functions
\[
g_n(s)= g\left(|s| + \tfrac{1}{n}\right) \quad \text{and} \quad h_n(s) = h\left(|s| + \tfrac{1}{n}\right),
\]
and we set $\Gamma_n(s):= \int_1^s g_n(t)\, \mathrm{d}t$. We consider the approximate problems
\begin{equation}
	\label{eq:PbApprox}
	\begin{cases}
		-\Delta_1 u_n = g_n(u_n)|Du_n| + h_n(u_n)f & \text{in }\Omega,\\
		u_n=0 & \text{on } \partial\Omega.
	\end{cases}
\end{equation}
Since $\limsup_{s\to +\infty} h_n(s) = h(\infty) <+ \infty$ by~\eqref{eq:hyp_h_inf} and $\|f\|_{L^{N,\infty}(\Omega)} < \left(\tilde{\mathcal{S}}_1 h(\infty)\right)^{-1}$, the existence of a non-negative solution $u_n\in \bvo \cap \lio$ to~\eqref{eq:PbApprox} with $D^j u_n = 0$ is guaranteed by Theorem~\ref{th:Appendix}. Furthermore, Theorem~\ref{th:Appendix} ensures that, if $M_n>0$ is any constant such that $\mathcal{\tilde S}_1^{-1} - \|f\|_{L^{N,\infty}(\Omega)} \sup_{s\geq M_n} h_n(s)>0$, then 
\[
\|u_n\|_\lio \leq M_n.
\]
In particular, taking $M>0$ such that $\mathcal{\tilde S}_1^{-1} - \|f\|_{L^{N,\infty}(\Omega)} \sup_{s\geq M} h(s)>0$ and noting that $\sup_{s\geq M} h_n(s) \leq \sup_{s\geq M} h(s)$, it follows that
\begin{equation}
\label{eq:lio_bound}
\|u_n\|_\lio \leq M,\ \forall n\in\N.
\end{equation}
This shows that the sequence $u_n$ is bounded in $\lio$. In what follows, we denote by $z_n\in \DM(\Omega)$ the vector field associated to $u_n$.

Taking $\phi \in \bvo\cap\lio$ as test function in the weak formulation of~\eqref{eq:PbApprox}, i.e., in \eqref{eq:app_weak_formA} one has 
\begin{equation}
\label{eq:app_weak_form}
\into (z_n, D\phi) - \intdo \phi [z_n,\nu] \ \dH = \into g_n(u_n) \phi |Du_n| + \into h_n(u_n) \phi f, \ \forall \phi \in \bvo\cap\lio.
\end{equation}
We note that, since $u_n$ has no jump part, $g_n(u_n)|Du_n| = |D\Gamma_n(u_n)|$.

In the following, we prove several uniform estimates for problem~\eqref{eq:PbApprox}. We first show that the solutions $u_n$ are uniformly bounded in $\bvo$.

\begin{lemma}  
\label{lem:bound_u_n}
  Let $0<f\in L^{N,\infty}(\Omega)$. Assume that $h$ satisfies~\eqref{eq:hyp_h_inf} and $\|f\|_{L^{N,\infty}(\Omega)} < \left(\tilde{\mathcal{S}}_1 h(\infty)\right)^{-1}$. Let $u_n$ be a solution to \eqref{eq:PbApprox}.  Then
\begin{equation*}
\|u_n\|_\bvo \leq C,\ \forall n\in\N,
\end{equation*}
for some positive constant $C$ independent of $n$.
\end{lemma}

\begin{proof}
We take $\phi=\left(u_n - \|u_n\|_\lio \right) \in \bvo \cap \lio$  in~\eqref{eq:app_weak_form} to obtain, as the right-hand side is non-positive, that
\[
\into \left(z_n, D \left(u_n - \|u_n\|_\lio \right)\right) - \intdo \left(u_n - \|u_n\|_\lio \right) [z_n, \nu]\ \dH \leq 0.
\]
Recalling that $(z_n, Du_n) = |Du_n|$ as measures in $\Omega$ and $u_n [z_n, \nu] = -u_n$ $\mathcal{H}^{N-1}$-a.e. on $\partial\Omega$, we deduce
\[
\into |Du_n| + \intdo u_n \ \dH \leq - \intdo \|u_n\|_\lio [z_n, \nu]\ \dH .
\]
Finally, using the boundedness of $u_n$ in $\lio$ (see~\eqref{eq:lio_bound}) and that $\|[z_n,\nu]\|_{L^\infty(\partial\Omega)}\leq 1$, we conclude that
\begin{equation*}
	\|u_n\|_{BV(\Omega)} \leq -\int_{\partial\Omega} \|u_n\|_{L^\infty(\Omega)} [z_n,\nu] \ \dH \leq \|u_n\|_{L^\infty(\Omega)} \mathcal{H}^{N-1}(\partial\Omega) \leq M \mathcal{H}^{N-1}(\partial\Omega),\ \forall n\in \N,
\end{equation*}
as desired.
\end{proof}

Next, we take advantage of the rigidity of the 1-Laplacian operator to obtain that the right-hand side terms are uniformly bounded as measures.

\begin{lemma}  
\label{lem:bound_Gamma_1}
 Let $0<f\in L^{N,\infty}(\Omega)$. Assume that $h$ satisfies~\eqref{eq:hyp_h_inf} and $\|f\|_{L^{N,\infty}(\Omega)} < \left(\tilde{\mathcal{S}}_1 h(\infty)\right)^{-1}$. Let $u_n$ be a solution to \eqref{eq:PbApprox}. 
Then
\begin{equation*}
\into |D\Gamma_n(u_n)| \leq C \quad \text{and} \quad \into h_n(u_n)f \leq C,\quad \forall n\in\N,
\end{equation*}
for some positive constant $C$ independent of $n$.
\end{lemma}

\begin{proof}
We take $\phi=1$ as test function in~\eqref{eq:app_weak_form}, obtaining that
\[
- \intdo [z_n,\nu] \ \dH = \into g_n(u_n) |Du_n| + \into h_n(u_n)f,\ \forall n\in\N.
\]
Since $\|[z_n,\nu]\|_{L^\infty(\partial\Omega)}\leq 1$, our claim readily follows.
\end{proof}

Our next objective is to show that the sequence $\Gamma_n(u_n)$ is bounded in $BV_{\rm loc}(\Omega)$. Given the above result, it suffices to establish that $\Gamma_n(u_n)$ is bounded in $L^1_{\rm loc}(\Omega)$. When $g$ satisfies the integrability condition~\eqref{eq:hyp_g_int}, this property follows immediately since, in fact, $\Gamma_n(u_n)$ remains bounded in $\lio$. However, if~\eqref{eq:hyp_g_int} is not satisfied, this approach fails as $\Gamma(0)=-\infty$. We shall demonstrate that this difficulty can be overcome provided that $f$ satisfies the local positivity~\eqref{eq:hyp_f_pos} and the interplay~\eqref{eq:hyp_sing} between $g$ and $h$ holds. The next results are directed toward this end.

First, we construct an auxiliary function.

\begin{lemma}
\label{lem:h_aux}
Assume that~\eqref{eq:hyp_sing} holds. Then there exists a non-increasing, locally Lipschitz function $\tilde h\colon (0,+\infty) \to [0,+\infty)$ such that $\tilde h(s) \leq h(s)$ for any $s>0$ and, for some $\eta>0$, verifies
\begin{equation}
\label{eq:in_aux}
    \limsup_{s\to 0^+} \frac{|\Gamma(s)|}{\tilde h(s)^\eta} <+  \infty.
\end{equation}
\end{lemma}

\begin{proof}
Observe that, thanks to~\eqref{eq:hyp_sing}, without loss of generality we may assume that $h>0$, as the argument away from zero is straightforward. The construction of $\tilde{h}$ then proceeds in two steps.
 
First, we define a smooth approximation of $\frac{1}{2}h$, denoted by $h_1$, which satisfies $\frac{1}{4}h\leq  h_1 \leq h$. This function can be obtained through a standard partition of unity argument. Specifically, since the family of intervals $\{(2^{n-1}, 2^{n+1})\}_{n\in\mathbb{Z}}$ constitutes a locally finite open cover of $(0,+\infty)$, a smooth partition of unity subordinated to this cover can be selected. By employing mollifiers, we construct a local smooth approximation of $\frac{1}{2}h$ that remains bounded between $\frac{1}{4}h$ and $h$ within each interval of the cover. Finally, these local approximations are glued together via the partition of unity to yield the desired  smooth  function $h_1$.

In the second phase, we define $\tilde h\colon (0,+\infty) \to [0,+\infty)$ as the lower envelope of $h_1$, that is, we set $\tilde h(s) = \inf_{t\leq s} h_1(t)$ for each $s>0$. By construction, $\tilde h$ is non-increasing and satisfies $\tilde h \leq h_1 \leq h$. Furthermore, $\tilde h$ is locally Lipschitz. To see this, fix an interval $I=[a_1,a_2]\subset (0,+\infty)$. Since $\lim_{s\to 0^+}h_1(s)=+\infty$ by~\eqref{eq:hyp_sing}, there exists $a_0\in (0,a_1]$ such that $\tilde h(s) = \min_{a_0 \leq t\leq s} h_1(t)$ for all $s\in I$. Given that $h_1$ is Lipschitz in $I$, it follows that $\tilde h$ inherits the same Lipschitz regularity in $I$. 

Finally, we verify that~\eqref{eq:in_aux} holds. On the one hand, since $h_1\geq \frac{1}{4}h$ and~\eqref{eq:hyp_sing} is satisfied, it follows that
\[
\tilde h(s) = \inf\nolimits_{t\leq s} h_1(t) \geq \frac{1}{4} s^{-\gamma},\ \forall s\leq s_0.
\]
On the other hand, by ~\eqref{eq:hyp_sing}, we have the estimate
\[
|\Gamma(s)|=\int_s^1 g(t) \, \mathrm{d}t \leq \int_s^1 t^{-\theta} \, \mathrm{d}t = \frac{1}{\theta-1} \left(s^{1-\theta}-1 \right),\ \forall s\leq \min\{1,s_0\}.
\]
Consequently, by fixing $\eta >\gamma^{-1} (\theta -1)$, we conclude that
\[
\limsup_{s\to 0^+} \frac{|\Gamma(s)|}{\tilde h(s)^\eta} \leq \frac{4^\eta}{\theta-1} \limsup_{s\to 0^+} s^{\gamma \eta}\left(s^{1-\theta} -1\right) = 0.
\]
This completes the proof.
\end{proof}

As before, we set $\tilde h_n(s) = \tilde h\left(s + \frac{1}{n}\right)$ for any $s\geq 0$. In the following, we obtain a bound for $\tilde h_n(u_n)$ in $L^m_{\rm loc} (\Omega)$ for any $m\geq 1$.

\begin{lemma}
\label{lem:h_L1_loc}
 Let $0<f\in L^{N,\infty}(\Omega)$. Assume that $h$ satisfies~\eqref{eq:hyp_h_inf} and $\|f\|_{L^{N,\infty}(\Omega)} < \left(\tilde{\mathcal{S}}_1 h(\infty)\right)^{-1}$. Moreover, suppose that $f$ verifies~\eqref{eq:hyp_f_pos} and that~\eqref{eq:hyp_sing} holds.  Let $u_n$ be a solution to \eqref{eq:PbApprox} and let $\tilde h$ be the function given by Lemma~\ref{lem:h_aux}. Then for any $m\in\N$ and any $\omega\subset\subset \Omega$ there exists some $C=C(m,\omega)>0$ such that
\begin{equation}
\label{eq:Pf_Lem_h_L1_b}
\int_\omega \tilde h_n(u_n)^m < C,\ \forall n\in \N.
\end{equation}
\end{lemma}

\begin{proof}
We are done if we prove that $\tilde h_n(u_n)^m$ is bounded in $L^1(B)$ for any ball $B$ compactly contained in $\Omega$. Let $B_R$ be a ball of radius $R>0$ compactly contained in $\Omega$. We can find $R'>R$ such that $B_{R} \subset \subset B_{R'} \subset \subset \Omega$, where $B_{R'}$ is a ball with the same centre as $B_{R}$.

We denote $\omega:= B_R$ and $\omega_0:=B_{R'}$. For all $m\in\N$, we define $\omega_m = B_{R+2^{-m}(R'-R)}$. Observe that $\omega\subset\omega_{m}\subset \omega_{m-1}\subset \omega_0$ for all $m \in \N$. We also define, for every $m\in \N$, a smooth cut-off function $0\leq \psi_m\in C_c^1(\Omega)$ such that $\psi_m\equiv 1$ in $\omega_m$ and $\psi_m\equiv 0$ in $\Omega\setminus \omega_{m-1}$.

Taking  $\psi_1$ as test function in~\eqref{eq:app_weak_form}, dropping the gradient term and using that $\|z_n\|_{\lio^N}\leq 1$, we deduce
\[
\into h_n(u_n)f \psi_1 \leq \into z_n\cdot \nabla\psi_1 \leq |\Omega| \cdot \|\nabla\psi_1\|_{L^\infty(\Omega)^N},\ \forall n\in \N.
\]
Since $f$ satisfies~\eqref{eq:hyp_f_pos} in $\omega_0$, $\psi_1\equiv 1$ in $\omega_1$, and $\tilde h_n(u_n)\leq h_n(u_n)$, it follows that
\[
\int_{\omega_1} \tilde h_n(u_n) \leq c_{\omega_0}^{-1} |\Omega|\cdot \|\nabla\psi_1\|_{L^\infty(\Omega)^N},\ \forall n\in \N.
\]

Now, we claim that, for all $m\in \N$, it holds
\begin{equation}
\label{eq:Pf_Lem_h_L1}
\int_{\omega_m} \tilde h_n(u_n)^m \leq c_{\omega_0}^{-m} |\Omega| \prod_{i=1}^m \|\nabla\psi_i\|_{L^\infty(\Omega)^N},\ \forall n\in \N.
\end{equation}

As previously established, this is true for $m=1$. Assume that our claim holds for some $m\in\N$ and let us show that it is also true for $m+1$. Since $\tilde h$ is locally Lipschitz and $u_n$ is bounded, it follows that $\tilde h_n(u_n)^m \in BV(\Omega)\cap\lio$. Consequently, we can take $\tilde h_n(u_n)^m \psi_{m+1}$ as test function in~\eqref{eq:app_weak_form}. Dropping the gradient term, and using that $\tilde h_n(u_n) \leq h_n(u_n)$, we obtain
\[
\into \tilde h_n(u_n)^{m+1} f \psi_{m+1} \leq \into \left(z_n, D\left(\tilde h_n(u_n)^m\psi_{m+1} \right) \right) = \into z_n\cdot \nabla \psi_{m+1} \tilde h_n(u_n)^m + \into (z_n,D\tilde h_n(u_n)^m) \psi_{m+1}, \ \forall n\in \N.
\]
As $\tilde h_n$ is locally Lipschitz and non-increasing, we have that $(z_n,D\tilde h_n(u_n)^m) = - |D\tilde h_n(u_n)^m|$ by Remark~\ref{rem:Composition_nonincr}, so we can drop the last term to obtain that
\[
\into \tilde h_n(u_n)^{m+1} f \psi_{m+1} \leq  \into z_n\cdot \nabla \psi_{m+1} \tilde h_n(u_n)^m, \ \forall n\in \N.
\]
Then, using~\eqref{eq:Pf_Lem_h_L1} we deduce that
\[
\into \tilde h_n(u_n)^{m+1} f \psi_{m+1} \leq  \into z_n\cdot \nabla \psi_{m+1} \tilde h_n(u_n)^m = \int_{\omega_m} z_n \cdot \nabla \psi_{m+1} \tilde h_n(u_n)^m \leq c_{\omega_0}^{-m} |\Omega| \prod_{i=1}^{m+1} \|\nabla\psi_i\|_{L^\infty(\Omega)^N},\ \forall n\in \N.
\]
Finally, recalling that $f$ satisfies~\eqref{eq:hyp_f_pos} in $\omega_0$ and $\psi_{m+1}\equiv 1$ in $\omega_{m+1}$, we obtain
\[
\int_{\omega_{m+1}} \tilde h_n(u_n)^{m+1} \leq c_{\omega_0}^{-(m+1)} |\Omega| \prod_{i=1}^{m+1} \|\nabla\psi_i\|_{L^\infty(\Omega)^N},\ \forall n\in \N,
\]
which is the desired inequality. Therefore,~\eqref{eq:Pf_Lem_h_L1} holds true for all $m\in\N$. Since $B_R=\omega\subset \omega_{m}$ for all $m\in\N$, we conclude that~\eqref{eq:Pf_Lem_h_L1_b} is satisfied.
\end{proof}

\begin{remark}
This estimate also works if $g\equiv 0$ and is new even in this case. Obviously, if $h$ is non-increasing and locally Lipschitz the entire argument can be made with $h$ instead of $\tilde h$, obtaining that $h(u)\in L^m_{\rm loc}(\Omega)$ for any $m\geq 1$ after an application of Fatou's Lemma. This is the case, for instance, of the model function $h(s)=s^{-\gamma}$.
\end{remark}
The preceding results allow us to establish that the sequence $\Gamma_n(u_n)$ is bounded in $BV_{\rm loc}(\Omega)$ under the appropriate assumptions.

\begin{lemma}
\label{lem:bound_Gamma_2}
Let $0<f\in L^{N,\infty}(\Omega)$. Assume that $h$ satisfies~\eqref{eq:hyp_h_inf} and $\|f\|_{L^{N,\infty}(\Omega)} < \left(\tilde{\mathcal{S}}_1 h(\infty)\right)^{-1}$. Let $u_n$ be a solution to \eqref{eq:PbApprox}.
Moreover, suppose that one of the following conditions holds:
\begin{enumerate}[i)]
    \item $g$ satisfies the integrability condition~\eqref{eq:hyp_g_int}.
    \item $f$ satisfies~\eqref{eq:hyp_f_pos} and~\eqref{eq:hyp_sing} holds.
\end{enumerate}
Then, for every $\omega\subset\subset \Omega$, there exists some $C=C(\omega)>0$ such that
\begin{equation*}
\|\Gamma_n(u_n)\|_{BV(\omega)} \leq C,\ \forall n\in\N.
\end{equation*}
\end{lemma}

\begin{proof}
First, observe that if the integrability condition~\eqref{eq:hyp_g_int} holds, the sequence $\Gamma_n(u_n)$ is bounded in $\lio$. Indeed, this follows from the fact that $\Gamma$ is a bounded function in $[0,M]$, where $M>0$ is the uniform bound of $u_n$ in $\lio$ established in~\eqref{eq:lio_bound}. Consequently, by Lemma~\ref{lem:bound_Gamma_1}, we conclude that the sequence $\Gamma_n(u_n)$ is bounded in $\bvo$.

On the other hand, if $f$ satisfies~\eqref{eq:hyp_f_pos} and~\eqref{eq:hyp_sing} hold, we can make use of Lemma~\ref{lem:h_aux} to ensure the existence of some constants $c,\tilde s>0$ and $n_0\in \N$ such that
\[
|\Gamma_n(u_n)| \leq c \tilde h_n(u_n)^\eta \chi_{\{u_n\leq \tilde s\}} + |\Gamma_n(u_n)| \chi_{\{u_n > \tilde s\}}, \ \forall n\geq n_0.
\]
By Lemma~\ref{lem:h_L1_loc}, the term $\tilde h_n(u_n)^\eta$ is uniformly bounded in $L^1_{\rm loc}(\Omega)$. Moreover, since $\Gamma$ is a bounded function in $[\tilde s,M]$, the term $|\Gamma_n(u_n)| \chi_{\{u_n > \tilde s\}}$ is uniformly bounded in $\lio$. These estimates, combined with Lemma~\ref{lem:bound_Gamma_1}, imply that $\Gamma_n(u_n)$ is a bounded sequence in $BV_{\rm loc}(\Omega)$.
\end{proof}

Now, we are in a position to prove the existence theorem. 

\begin{proof}[Proof of Theorem~\ref{th:Exist}]

By virtue of Lemma~\ref{lem:bound_u_n}, there exists a subsequence of $u_n$ (not relabelled) and some non-negative $u\in \bvo$ such that
\begin{equation*}
    u_n\to u \text{ in } L^1(\Omega) \qquad \text{ and } \qquad u_n\to u \ \ae \text{ in } \Omega.
\end{equation*}

Moreover, since $\into h_n(u_n)f$ is bounded by Lemma~\ref{lem:bound_Gamma_1}, the Fatou Lemma implies that $h(u)f\in L^1(\Omega)$. Given that $f>0$, it follows that $u>0$ $\ae$ in $\Omega$ if $h(0)=+\infty$. It is worth noting that, in our framework, this condition is always satisfied when the integrability condition~\eqref{eq:hyp_g_int} does not hold, as in that case we assume the singular behaviour~\eqref{eq:hyp_sing}.

In addition, from Lemma~\ref{lem:bound_Gamma_2} we deduce that $\Gamma_n(u_n)$ has a subsequence (not relabelled) such that $\Gamma_n(u_n)$ converges in $L^1_{\rm loc}(\Omega)$ and $\ae$ in $\Omega$ to some function belonging to $BV_{\rm loc}(\Omega)$. As $u_n\to u$ $\ae$ in $\Omega$, the limit of $\Gamma_n(u_n)$ must be $\Gamma(u)$. While this is immediate when $g$ is integrable, in the non-integrable case we rely on the fact that $u>0$ $\ae$ in $\Omega$ since~\eqref{eq:hyp_sing} holds. Therefore, $\Gamma(u)\in BV_{\rm loc}(\Omega)$.

Finally, since the sequence $e^{\Gamma_n(u_n)}$ is bounded in $\lio$ (as a consequence of the boundedness of $u_n$) and $\Gamma_n(u_n)\to \Gamma(u)$ $\ae$ in $\Omega$, the Lebesgue Theorem implies that
\begin{equation}
\label{eq:conv_e^Gamma}
    e^{\Gamma_n(u_n)} \to e^{\Gamma(u)} \text{ in } L^q(\Omega) \text{ for all } 1\leq q <+ \infty.
\end{equation}

In the following, for the sake of clarity, we divide the proof into several steps.
\medskip 

\textbf{Step 1.} Existence of the vector field $z$.
\smallskip

As $\|z_n\|_{L^\infty(\Omega)^N}\le 1$, then there exists $z\in L^\infty(\Omega)^N$ such that $z_n\rightharpoonup z$ *-weakly in $L^\infty(\Omega)^N$. Since the norm is *-weakly lower semicontinuous, then we also have $\|z\|_{L^\infty(\Omega)^N}\leq 1$.
\medskip

\textbf{Step 2.} The vector field $z$ is in $\DM(\Omega)$.
\smallskip

Taking a non-negative $\varphi\in C_c^1(\Omega)$ as test function in~\eqref{eq:app_weak_form}, we obtain
\[
\into z_n \cdot \nabla \varphi = \into \varphi |D\Gamma_n(u_n)| + \into h_n(u_n)f \varphi.
\]
In the first integral, the limit is readily obtained by using that $z_n\rightharpoonup z$ *-weakly in $L^\infty(\Omega)^N$. For the second term, we exploit the lower semicontinuity of the functional $w \mapsto \into \varphi |Dw|$ on $\bvo$ with respect to the $L^1$-convergence; note that $\Gamma_n(u_n)\to \Gamma(u)$ locally in $L^1(\Omega)$. Finally, the last term is handled via Fatou's Lemma. Then, we deduce that
\begin{equation*}
\into z \cdot \nabla \varphi \geq \into \varphi |D\Gamma(u)| + \into h(u)f \varphi,
\end{equation*}
and thus
\begin{equation}
    \label{eq:prov8}
    -\operatorname{div} z \geq |D\Gamma(u)| + h(u)f \text{ in } \mathcal{D}'(\Omega).
\end{equation}
Therefore,   by Lemma~\ref{lem:Global_z}, one has $z\in \DM(\Omega)$.

\medskip
\textbf{Step 3.} An auxiliary distributional  inequality. 
\smallskip

We claim that 
\begin{equation}
\label{eq:prov1}
-\operatorname{div} \left( e^{\Gamma(u)}z \right) \leq h(u) f e^{\Gamma(u)} \text{ as measures in } \Omega.
\end{equation}
To prove~\eqref{eq:prov1}, we distinguish between two cases: $h(0)<+\infty$ and $h(0)=+\infty$.

Suppose first that $h(0)<+\infty$. Then, given $0\leq \varphi \in C_c^1(\Omega)$, we can take $e^{\Gamma_n(u_n)} \varphi \in \bvo\cap \lio$ as test function in~\eqref{eq:app_weak_form} to obtain
\begin{equation}
\label{eq:prov2}
\into \varphi \left(z_n, De^{\Gamma_n(u_n)}\right) + \into e^{\Gamma_n(u_n)} z_n \cdot \nabla \varphi = \into g_n(u_n) e^{\Gamma_n(u_n)} \varphi |Du_n| + \into h_n(u_n) e^{\Gamma_n(u_n)} f \varphi.
\end{equation}
By virtue of Lemma~\ref{lem:Composition}, it follows that $\left(z_n, De^{\Gamma_n(u_n)} \right) = \left| De^{\Gamma_n(u_n)} \right|$ and, since $u_n$ has no jump part, it follows from Lemma \ref{lem:ChainRule_AFP} that $\left| De^{\Gamma_n(u_n)} \right| = g_n(u_n) e^{\Gamma_n(u_n)} |Du_n|$. Cancelling terms in~\eqref{eq:prov2}, it yields
\begin{equation}
\label{eq:prov3}
\into e^{\Gamma_n(u_n)} z_n \cdot \nabla \varphi = \into h_n(u_n) e^{\Gamma_n(u_n)} f \varphi.
\end{equation}
To pass to the limit in~\eqref{eq:prov3}, we exploit the *-weak convergence of $z_n$ in $L^\infty(\Omega)^N$, the strong convergence of $e^{\Gamma_n(u_n)}$ provided by~\eqref{eq:conv_e^Gamma}, the boundedness of $h$, which follows from $h(0)<+\infty$ and assumption~\eqref{eq:hyp_h_inf}. Therefore, passing to the limit in~\eqref{eq:prov3} we arrive at
\begin{equation*}
\into e^{\Gamma(u)} z \cdot \nabla \varphi = \into h(u) e^{\Gamma(u)} f \varphi
\end{equation*}
and then~\eqref{eq:prov1} holds.

On the contrary, suppose now that $h(0)=+\infty$. In this case, observe that, as previously noted, $u>0$ $\ae$ in $\Omega$, which follows from the fact that $h(u)f\in L^1(\Omega)$ and $f>0$. Let $0\leq \varphi \in C_c^1(\Omega)$, let $\delta>0$, and let us define the non-decreasing function $S_\delta(s) = \min \big\{1,\max \big\{0,\frac{s-\delta}{\delta}\big \}\big\}$ for all $s\geq 0$. Taking $S_\delta(u_n) e^{\Gamma_n(u_n)} \varphi \in\bvo\cap\lio$ as test function in~\eqref{eq:app_weak_form} and using the Leibniz rule~\eqref{eq:Prod}, it yields
\begin{equation}
\label{eq:prov4}
\begin{split}
\into S_\delta(u_n) \varphi \left(z_n, De^{\Gamma_n(u_n)}\right) &+ \into e^{\Gamma_n(u_n)} \varphi \left(z_n, DS_\delta(u_n) \right) + \into e^{\Gamma_n(u_n)} S_\delta(u_n) z_n \cdot \nabla \varphi \\
&= \into g_n(u_n) S_\delta(u_n) e^{\Gamma_n(u_n)} \varphi |Du_n| + \into h_n(u_n) S_\delta(u_n) e^{\Gamma_n(u_n)} f \varphi.
\end{split}
\end{equation}
We now define the auxiliary functions
\[
\Psi_{\delta}(s) = \int_0^s S_\delta'(t) e^{\Gamma(t)} \, \mathrm{d}t,\ \forall s\geq 0, \qquad \text{ and } \qquad \Psi_{\delta,n}(s) = \int_0^s S_\delta'(t) e^{\Gamma_n(t)} \, \mathrm{d}t,\ \forall s\geq 0.
\] 
As previously argued, the identity $\left(z_n, De^{\Gamma_n(u_n)} \right) = g_n(u_n) e^{\Gamma_n(u_n)} |Du_n|$ allows us to cancel two terms in~\eqref{eq:prov4}. Furthermore, Remark~\ref{rem:Composition_nondecr}  ensures that $(z_n, DS_\delta(u_n)) = |DS_\delta(u_n)|$. Given that $u_n$ has no jump part, it follows that $|DS_\delta(u_n)|e^{\Gamma_n(u_n)} = |D\Psi_{\delta,n}(u_n)|$. Therefore, identity~\eqref{eq:prov4} can be rewritten as
\begin{equation}
\label{eq:prov5}
\into \varphi |D\Psi_{\delta,n}(u_n)| + \into e^{\Gamma_n(u_n)} S_\delta(u_n) z_n \cdot \nabla \varphi = \into h_n(u_n) S_\delta(u_n) e^{\Gamma_n(u_n)} f \varphi.
\end{equation}
In the following, we pass to the limit in~\eqref{eq:prov5} when $n\to+\infty$. Regarding the first term, we exploit the lower semicontinuity of the functional $w \mapsto \into \varphi |Dw|$ defined in $\bvo$ with respect to the $L^1$-convergence; note that $\Psi_{\delta,n}(u_n)\to \Psi_\delta(u)$ in $L^1(\Omega)$ when $n$ diverges. In the second integral, one can take advantage of the fact that ${z_n\rightharpoonup z}$ *-weakly in $L^\infty(\Omega)^N$. With respect to the last term, observe that for every $n\in\N$, the product $h_n(s)S_\delta(s)$ vanishes for $s<\delta$ and, as a consequence, $h_n(u_n) S_\delta(u_n) e^{\Gamma_n(u_n)}$ is uniformly bounded in $\lio$, allowing for the application of the Lebesgue Theorem. In this way, we obtain that
\begin{equation}
\label{eq:prov6}
\into \varphi |D\Psi_{\delta}(u)| + \into e^{\Gamma(u)} S_\delta(u) z \cdot \nabla \varphi \leq \into h(u) S_\delta(u) e^{\Gamma(u)} f \varphi.
\end{equation}
Next, we pass to the limit in~\eqref{eq:prov6} as $\delta\to 0^+$. With respect to the first term of~\eqref{eq:prov6}, we use again the lower semicontinuity of the functional $w \mapsto \into \varphi |Dw|$; note that $\Psi_\delta(u)\to e^{\Gamma(0)}\chi_{\{u>0\}}$ in $L^1(\Omega)$ when $\delta\to 0^+$. In the remaining integrals, the Lebesgue Theorem can be applied, since $h(u)f\in L^1(\Omega)$ and $e^{\Gamma(u)}\in \lio$. Hence, it yields
\begin{equation}
\label{eq:prov7}
\into \varphi e^{\Gamma(0)} \left| D\chi_{\{u>0\}} \right| + \into e^{\Gamma(u)} \chi_{\{u>0\}} z \cdot \nabla \varphi \leq \into h(u) \chi_{\{u>0\}} e^{\Gamma(u)} f \varphi.
\end{equation}
Finally, since $u>0$ $\ae$ in $\Omega$, the inequality~\eqref{eq:prov7} implies that~\eqref{eq:prov1} holds.

\medskip
\textbf{Step 4.} The limit $u$ has no jump part.
\smallskip

Here, we use the alternative pairing defined in~\eqref{eq:new_Pairing}. Following the notation introduced therein, we set $\beta(s) = e^s$, $v=\Gamma(u)$, and consider the pairing $\left( z, D\beta(v)^\# \right)$. Henceforth, when we write $\left(z, D\left( e^{\Gamma(u)} \right)^\# \right)$ we mean the pairing $\left( z, D\beta(v)^\# \right)$. It then follows that, as measures in $\Omega$,
\begin{equation}
\label{eq:prov9}
\begin{split}
    \left(e^{\Gamma(u)} \right)^\# |D\Gamma(u)| & \overset{\eqref{eq:prov8}}{\leq} -\left(e^{\Gamma(u)} \right)^\# \operatorname{div}z - e^{\Gamma(u)} h(u)f \overset{\eqref{eq:prov1}}{\leq} - \left(e^{\Gamma(u)} \right)^\# \operatorname{div}z + \operatorname{div} \left( e^{\Gamma(u)}z \right)\\
    & \overset{\eqref{eq:new_Pairing}}{=} \left(z, D\left( e^{\Gamma(u)} \right)^\# \right) \leq \left|De^{\Gamma(u)} \right| \overset{\eqref{eq:chain_rule_2}}{=} \left(e^{\Gamma(u)} \right)^\# |D\Gamma(u)|.
\end{split}
\end{equation}
In this way, all the inequalities in~\eqref{eq:prov9} are equalities. In particular, we have that
\begin{equation}
\label{eq:prov14}
    \left(z, D\left( e^{\Gamma(u)} \right)^\# \right) = \left|De^{\Gamma(u)} \right| \text{ as measures in } \Omega.
\end{equation}
We can then use Lemma~\ref{lem:no_jump} in~\eqref{eq:prov8} with $\alpha(s)=s$, $\beta(s)=e^s$ and $v=\Gamma(u)$ to deduce that $D^j \Gamma(u)=0$. Since $\Gamma$ is increasing, we conclude that $D^j u=0$.
\medskip

\textbf{Step 5.} Identification of the vector field $z$ given by~\eqref{eq:def_pair}.
\smallskip

Since $u$ has no jump part and the mapping $s\mapsto e^{\Gamma(s)}$ is increasing, it follows that $e^{\Gamma(u)}$ also has no jump part. Therefore, the identity $\left( e^{\Gamma(u)} \right)^\# = e^{\Gamma(u)}$ holds $\mathcal{H}^{N-1}$-$\ae$ in $\Omega$ and~\eqref{eq:prov14} can be rewritten as
\[
\left(z, De^{\Gamma(u)} \right) = \left|De^{\Gamma(u)} \right| \text{ as measures in } \Omega.
\]
Since $s\mapsto e^{\Gamma(s)}$ is an increasing mapping and $u\in \bvo\cap\lio$, Lemma~\ref{lem:Composition} ensures that the identity~\eqref{eq:def_pair} holds.

\medskip

\textbf{Step 6.} Distributional formulation~\eqref{eq:def_dist}.
\smallskip

From~\eqref{eq:prov9}, recalling that $u$ has no jump part, we deduce that
\begin{equation}
\label{eq:prov18}
   e^{\Gamma(u)} \left(-\operatorname{div} z - |D\Gamma(u)| - h(u)f \right) = 0 \text{ as measures in } \Omega.
\end{equation}
We claim that $e^{\Gamma(u)}$ is strictly positive $\mathcal{H}^{N-1}$-a.e. in $\Omega$. This is immediate when $\Gamma(0)$ is finite, since in that case $e^{\Gamma(u)}\geq e^{\Gamma(0)}>0$. In the non-integrable case, it follows from the fact that $\Gamma(u)\in BV_{\rm loc}(\Omega)$, since the precise representative of a $BV_{\rm loc}$ function is finite $\mathcal{H}^{N-1}$-a.e. in $\Omega$. Since $\Gamma$ is increasing and $D^j u = 0$, at every point $x$ where $u^*(x) = 0$, the function $\Gamma(u)$ has approximate limit $\Gamma(0) = -\infty$. Consequently, we deduce that
\begin{equation}
\label{eq:prov19}
\mathcal{H}^{N-1}\left(\left\{ x\in\Omega : u^*(x)=0 \right\}\right) = 0,
\end{equation}
which establishes the claim.

Let us now denote $\mu:= -\operatorname{div} z - |D\Gamma(u)| - h(u)f$, which is a non-negative measure by~\eqref{eq:prov8}. On the one hand, \eqref{eq:prov18} implies that $\mu$ vanishes on the set where $e^{\Gamma(u)}$ is positive. On the other hand, $0\leq \mu \leq -\operatorname{div}z$ and $\operatorname{div}z$ is absolutely continuous with respect to $\mathcal{H}^{N-1}$ (see~\cite[Proposition~3.1]{ChenFrid}), so that $\mu$ does not charge $\mathcal{H}^{N-1}$-negligible sets. Since $e^{\Gamma(u)}>0$ $\mathcal{H}^{N-1}$-$\ae$ in $\Omega$, we conclude that $\mu=0$, that is,
\begin{equation}
\label{eq:prov16}
    -\operatorname{div} z = |D\Gamma(u)| + h(u)f \text{ as measures in } \Omega.
\end{equation}
Finally, since $u$ has no jump part, Lemma~\ref{lem:ChainRule} implies that $\chi_{\{u>0\}}|D\Gamma(u)| = g(u) |Du|$, and the identity $|D\Gamma(u)| = g(u) |Du|$ follows at once when $\Gamma(0)$ is finite. In the non-integrable case, it is enough to observe that $|D\Gamma(u)|\leq -\operatorname{div}z$ by~\eqref{eq:prov8}, so that $|D\Gamma(u)|$ is absolutely continuous with respect to $\mathcal{H}^{N-1}$ and therefore, by~\eqref{eq:prov19}, it does not charge $\{u=0\}$.

\medskip

\textbf{Step 7.} Boundary condition~\eqref{eq:def_bord}.
\smallskip

Let $\phi$ be a non-negative Lipschitz function on $\partial\Omega$. Then there exists $\varphi \in C(\overline{\Omega})$ which is a non-negative Lipschitz function on $\Omega$ and such that $\varphi\big|_{\partial\Omega} = \phi$. 

Fixing $\delta >0$ and $m\in \N$, we choose $G_\delta(u_n)^m\varphi \in BV(\Omega)\cap \lio$ as test function in~\eqref{eq:app_weak_form}, where $G_\delta$ is defined in~\eqref{eq:def_Tk_Gk}. Applying the Leibniz rule~\eqref{eq:Prod}, we obtain
\begin{equation*}
\begin{split}
\into \varphi \left(z_n, DG_\delta(u_n)^m \right) &+ \into G_\delta(u_n)^m z_n \cdot \nabla \varphi - \intdo G_\delta(u_n)^m \phi [z_n,\nu] \ \dH\\
&= \into g_n(u_n) G_\delta(u_n)^m \varphi |Du_n| + \into h_n(u_n) G_\delta(u_n)^m f \varphi.
\end{split}
\end{equation*}
By Remark \ref{rem:Composition_nondecr}, we have the identity $\left(z_n, DG_\delta(u_n)^m \right) = \left|DG_\delta(u_n)^m \right|$. Moreover, the equality $u_n(1+[z_n,\nu])=0$ $\mathcal{H}^{N-1}$-$\ae$ on $\partial\Omega$ implies that $G_\delta(u_n)^m[z_n,\nu]=-G_\delta(u_n)^m$ $\mathcal{H}^{N-1}$-$\ae$ on $\partial\Omega$, since $G_\delta(0)=0$. Recalling that $u_n$ has no jump part, we obtain that
\begin{equation}
\label{eq:prov15}
\begin{split}
m\into G_\delta(u_n)^{m-1} \varphi | Du_n | &+ \into G_\delta(u_n)^m z_n \cdot \nabla \varphi + \intdo G_\delta(u_n)^m \phi \ \dH\\
&= \into g_n(u_n) G_\delta(u_n)^m \varphi |Du_n| + \into h_n(u_n) G_\delta(u_n)^m f \varphi.
\end{split}
\end{equation}

Now, we fix $M>0$ such that $\|u_n\|_\lio \leq M$ for all $n\in \N$ (see~\eqref{eq:lio_bound}). For each $s\in [0,M]$ and $n\in \N$,  we define the auxiliary functions
\begin{equation*}
\tilde \Gamma_\delta(s):= \int_0^s \big(m-g(t) G_\delta(t)\big) G_\delta(t)^{m-1} \, \mathrm{d}t \qquad \text{and} \qquad \tilde \Gamma_{\delta,n}(s):= \int_0^s \big(m-g_n(t) G_\delta(t)\big) G_\delta(t)^{m-1} \, \mathrm{d}t.
\end{equation*}
From now on, we consider $m$ sufficiently large such that $\tilde \Gamma_\delta, \tilde \Gamma_{\delta, n}$ are non-decreasing. Using again that $u_n$ has no jump part, and taking into account that $\tilde\Gamma_{\delta, n}(s) \leq G_\delta(s)^m$ for all $s\in [0,M]$, we deduce from~\eqref{eq:prov15} that 
\begin{equation}
\label{eq:prov11}
\into \varphi | D \tilde\Gamma_{\delta,n}(u_n) | + \intdo \tilde\Gamma_{\delta,n}(u_n) \phi \ \dH \leq  \into h_n(u_n) G_\delta(u_n)^m f \varphi - \into G_\delta(u_n)^m z_n \cdot \nabla \varphi.
\end{equation}
To pass to the limit in~\eqref{eq:prov11} as $n\to+\infty$, we invoke the lower semicontinuity of the functional $w \mapsto \into \varphi|Dw| + \intdo |w| \varphi$ defined in $BV(\Omega)$ with respect to the $L^1(\Omega)$-convergence; note that $\tilde \Gamma_{\delta,n}(u_n) \to \tilde \Gamma_\delta(u)$ in $L^1(\Omega)$ as $n$ diverges. Regarding the third integral, the Lebesgue Theorem is applicable since $h_n(s)G_\delta(s)$ vanishes for $s<\delta$, ensuring that $h_n(u_n) G_\delta(u_n)^m$ is uniformly bounded in $\lio$. Finally, also using the *-weak convergence in $L^\infty(\Omega)^N$ of $z_n$, the passage to the limit in~\eqref{eq:prov11} yields 
\begin{equation}
\label{eq:prov12}
\into \varphi |D\tilde \Gamma_\delta(u)| + \intdo \tilde \Gamma_\delta(u) \phi \ \dH \leq \into h(u) G_\delta(u)^m f \varphi - \into G_\delta(u)^m z  \cdot \nabla \varphi.
\end{equation}
Inequality~\eqref{eq:prov12} can be simplified as follows. Since $G_\delta(u)^m \in \lio$, from~\eqref{eq:prov16} one deduces that
\begin{equation}
\label{eq:prov17}
-G_\delta(u)^m \operatorname{div}z = g(u) G_\delta(u)^m |Du| + h(u)G_\delta(u)^m f \text{ as measures in } \Omega.
\end{equation}
We now substitute~\eqref{eq:prov17} into~\eqref{eq:prov12}. Taking into account the identity $|D\tilde \Gamma_\delta(u)| = |DG_\delta(u)^m| - g(u)G_\delta(u)^m |Du|$ and cancelling terms, we find that~\eqref{eq:prov12} reduces to
\begin{equation*}
\into \varphi |DG_\delta(u)^m| + \intdo \tilde\Gamma_\delta(u) \phi \ \dH \leq - \into G_\delta(u)^m \varphi \operatorname{div}z - \into G_\delta(u)^m z  \cdot \nabla \varphi.
\end{equation*}
Using the Green identity~\eqref{eq:Green} and the Leibniz rule~\eqref{eq:Prod}, we obtain that
\begin{equation*}
\into \varphi |DG_\delta(u)^m| + \intdo \tilde\Gamma_\delta(u) \phi \ \dH \leq \into \varphi \left(z, DG_\delta(u)^m \right) - \intdo G_\delta(u)^m \phi [z,\nu] \ \dH.
\end{equation*}
Since $(z,DG_\delta(u)^m) = |DG_\delta(u)^m|$ by Remark \ref{rem:Composition_nondecr}, we conclude that  
\begin{equation*}
\intdo \phi \left(\tilde\Gamma_\delta(u) + G_\delta(u)^m [z,\nu] \right) \ \dH = \intdo \phi \left(G_\delta(u)^m(1 + [z,\nu]) - \int_0^u g(t)G_\delta(t)^m\, \mathrm{d}t \right) \ \dH \leq 0,
\end{equation*}
which, by an immediate density argument, holds for every non-negative $\phi\in C(\partial\Omega)$. This implies the existence of a subset $S\subset\partial\Omega$ with $\mathcal{H}^{N-1}(S)=0$ such that, for every $\delta>0$, it holds
\[
G_\delta(u(x))^m(1 + [z,\nu](x)) - \int_0^{u(x)} g(t)G_\delta(t)^m\, \mathrm{d}t \leq 0,\ \forall x\in\partial\Omega \setminus S.
\]

For any $x\in \partial\Omega \setminus S$ such that $u(x)>0$, we may choose $\delta>0$ small enough ($0<\delta<u(x)$) to obtain
\begin{equation*}
1 + [z,\nu](x) \leq \frac{\int_0^{u(x)} g(t) G_\delta(t)^m\, \mathrm{d}t}{G_\delta(u(x))^m} \leq \frac{\max_{s\in [\delta,M]}g(s)}{m+1} G_\delta(u(x)), \forall m\gg 1.
\end{equation*}
Letting $m\to+\infty$, it follows that $1 + [z,\nu](x)=0$. This argument shows that for each $x\in \partial\Omega \setminus S$, either $u(x)=0$ or $1 + [z,\nu](x)=0$. Therefore, we have shown that $u(x) (1 + [z,\nu](x)) = 0$ for $\mathcal{H}^{N-1}$-$\ae$ $x\in\partial\Omega$, which concludes the proof.
\end{proof}

\section{Non-zero trace solutions}
\label{sec:Trace}

This section is devoted to the proof of Theorem~\ref{th:Trace}, where we show that no solution to~\eqref{eq:PbMain} with zero boundary trace can exist if $g$ does not satisfy~\eqref{eq:hyp_g_int}.

\begin{proof}[Proof of Theorem~\ref{th:Trace}]
Let $u\in BV(\Omega)$ be a solution to~\eqref{eq:PbMain}. Note that $u$ is positive a.e. in $\Omega$ since $\Gamma(u)\in BV_{\rm loc}(\Omega)$ and $\Gamma(0)=-\infty$. We define the auxiliary function
\begin{equation*}
	\tilde g(s) = \begin{cases}
		g(s), & 0<s<1,\\
		\min \left\{g(s), s^{-2} \right\}, & s\geq 1.
	\end{cases}
\end{equation*}
Observe that $\tilde g(s) \leq g(s)$ for any $s>0$ and that $\tilde g(s)$ is integrable at infinity; in fact, $\int_1^\infty \tilde g(t) \, \mathrm{d}t\leq 1$. We denote 
\[
\tilde \Gamma(s) := \int_1^s \tilde g(t)\, \mathrm{d}t,\ \forall s>0 \qquad \text{ and } \qquad 
\tilde \Gamma_n(s) := \int_1^s T_n(\tilde g(t))\, \mathrm{d}t,\ \forall s\geq 0, \ \forall n\in\N.
\] 
We point out that each $\tilde \Gamma_n$ is increasing and belongs to $C^1([0,+\infty))$. Furthermore, it holds that $\lim_{s\to+\infty} \tilde \Gamma_n(s) \leq \int_1^\infty \tilde g(t) \, \mathrm{d}t \leq 1$ for every $n\in\N$ and, since $g(s)$ is not integrable at $s=0$, we also have $\lim_{n\to+\infty} \tilde \Gamma_n(0) = -\infty$.

We take $e^{\tilde \Gamma_n(u)} \in BV(\Omega) \cap L^\infty(\Omega)$ as test function in~\eqref{eq:PbMain} to obtain, after using Green's identity~\eqref{eq:Green}, that
\[
\into \left(z,De^{\tilde \Gamma_n(u)}\right) - \intdo e^{\tilde \Gamma_n(u)}[z,\nu] \ \dH = \into g(u) e^{\tilde \Gamma_n(u)} |Du| + \into h(u)e^{\tilde \Gamma_n(u)} f.
\]
Since $s\mapsto e^{\tilde \Gamma_n(s)}$ is an increasing function, using Lemma~\ref{lem:Composition} we obtain that
\[
\into \left| De^{\tilde \Gamma_n(u)} \right| - \intdo e^{\tilde \Gamma_n(u)}[z,\nu] \ \dH = \into g(u) e^{\tilde \Gamma_n(u)} |Du| + \into h(u) e^{\tilde \Gamma_n(u)} f.
\]
As $u$ has no jump part, we can use the chain rule to obtain
\begin{equation}
    \label{eq:tr1}
    \into T_n(\tilde g(u)) e^{\tilde \Gamma_n(u)} |Du| - \intdo e^{\tilde \Gamma_n(u)} [z,\nu] \ \dH = \into g(u) e^{\tilde \Gamma_n(u)} |Du| + \into h(u) e^{\tilde \Gamma_n(u)} f.
\end{equation}

In the following, we want to pass to the limit in~\eqref{eq:tr1}. Observe that $e^{\tilde \Gamma_n(s)} \to e^{\tilde \Gamma(s)}$ as $n\to +\infty$ for any $s\geq 0$, where $e^{\tilde \Gamma(0)}:=0$ since $\tilde \Gamma(0)=-\infty$. Moreover, $e^{\tilde \Gamma_n(s)}$ is uniformly bounded because $e^{\tilde \Gamma_n(s)} \leq e^{\lim_{t\to+\infty} \tilde \Gamma_n(t)} \leq e$ for any $s\geq 0$ and $n\in\N$. As a consequence, $e^{\tilde \Gamma_n(u)}$ is bounded in $\lio$ and $L^\infty(\partial\Omega)$, and satisfies that $e^{\tilde \Gamma_n(u)} \to e^{\tilde \Gamma(u)}$ $\ae$ in $\Omega$ and $\mathcal{H}^{N-1}$-$\ae$ on $\partial\Omega$. Finally, note that $T_n(\tilde g(u)) \to \tilde g(u)$ $\ae$ in $\Omega$, $T_n(\tilde g(u))\leq \tilde g(u)$ and $g(u), \tilde g(u) \in L^1(\Omega,|Du|)$ (see Remark~\ref{rem:g|Du|}). Using the Lebesgue Theorem, one can pass to the limit to deduce
\begin{equation}
    \label{eq:tr2}
    \into \tilde g(u)e^{\tilde \Gamma(u)} |Du| - \intdo e^{\tilde \Gamma(u)} [z,\nu] \ \dH = \into g(u) e^{\tilde \Gamma(u)} |Du| + \into h(u) e^{\tilde \Gamma(u)} f.
\end{equation}

Since $e^{\tilde \Gamma(0)}=0$, from the boundary condition~\eqref{eq:def_bord} we deduce that $e^{\tilde \Gamma(u)}[z,\nu] = -e^{\tilde \Gamma(u)}$ $\mathcal{H}^{N-1}$-a.e. on $\partial\Omega$. Taking into account that $\tilde g(s) \leq g(s)$ for all $s>0$, we drop non-negative terms to obtain from~\eqref{eq:tr2} that 
\[
\intdo e^{\tilde \Gamma(u)} \ \dH \geq \into h(u) e^{\tilde \Gamma(u)} f.
\]
We underline that $e^{\tilde \Gamma(s)}=0$ if and only if $s=0$. Therefore, since $u>0$ $\ae$ in $\Omega$ and $h$ is a positive function, one has $\into h(u) e^{\tilde \Gamma(u)} f>0$. From this, we obtain that $\intdo e^{\tilde \Gamma(u)} \ \dH >0$. In this way, we conclude that there exists some $S\subset \partial\Omega$ with $\mathcal{H}^{N-1}(S)>0$ such that $u(x)>0$ for any $x\in S$.
\end{proof}

\begin{remark}
If $g$ is integrable at infinity, there is no need to introduce $\tilde g$: the entire argument can be carried out by replacing $\tilde g$ with $g$.
\end{remark}

\begin{corollary}
\label{cor:Trace}
Let $0<f\in L^1(\Omega)$, and assume that $g$ does not satisfy~\eqref{eq:hyp_g_int} and $h$ is positive.  If $\Omega$ is a ball, then any radially symmetric solution $u$ to~\eqref{eq:PbMain} satisfies $u\lvert_{\partial\Omega}>0$.
\end{corollary}

\section{On the sharpness of the main results}
\label{sec:Sharpness} 

This section is devoted to discussing the necessity of some of the assumptions in the main results of Section~\ref{sec:Main}.

Regarding the smallness condition on $f$ in Theorem~\ref{th:Exist}, note that it is formulated for data $f$ belonging to $L^{N,\infty}(\Omega)$, a space strictly larger than $L^N(\Omega)$ (for instance, $|x-x_0|^{-1} \in L^{N,\infty}(\Omega) \setminus L^N(\Omega)$ for any $x_0 \in \Omega$). However, this broader framework imposes a trade-off: the smallness condition $\|f\|_{L^{N,\infty}(\Omega)} < \left(\tilde{\mathcal{S}}_1 h(\infty)\right)^{-1}$ involves the constant $\tilde{\mathcal{S}}_1$, which comes from the embedding $BV(\Omega) \hookrightarrow L^{\frac{N}{N-1},1}(\Omega)$.

When $f\in L^N(\Omega)$, one can instead employ the standard Sobolev embedding $BV(\Omega) \hookrightarrow L^{\frac{N}{N-1}}(\Omega)$, whose sharp constant is $\mathcal{S}_1$. Consequently, the alternative condition
\begin{equation}
\label{eq:smallness_ln}
\|f\|_{L^{N}(\Omega)} < \left(\mathcal{S}_1 h(\infty)\right)^{-1}
\end{equation}
also guarantees the existence of solution.

In general, for $f\in L^N(\Omega)$ one only has $\|f\|_{L^{N,\infty}(\Omega)} \leq \|f\|_{L^N(\Omega)}$, which means that neither condition implies the other when $h(\infty)>0$, so both must be taken into account. Nevertheless, for functions in $L^N(\Omega)$ whose $L^N$ and $L^{N,\infty}$ norms coincide (such as constant functions), condition~\eqref{eq:smallness_ln} is strictly less restrictive than the $L^{N,\infty}(\Omega)$ condition, due to the fact that $\mathcal{S}_1^{-1} = N\omega_N^{1/N} > \tilde{\mathcal{S}}_1^{-1} = (N-1)\omega_N^{1/N}$.

The next example, adapted from~\cite[Example~3]{DaSe}, shows the necessity of the smallness condition in Theorem~\ref{th:Exist}. It also demonstrates that the threshold~\eqref{eq:smallness_ln} is sharp for the existence of a solution to~\eqref{eq:PbMain} when $h(\infty) = \inf_{s\geq 0} h(s)$.

\begin{example}[Non-existence of solution]
Let $h$ be such that $\inf_{s\geq 0} h(s)>0$, and let $f=\lambda \chi_{B_R(x_0)}$, where $x_0\in\Omega$ and $R>0$ are chosen such that $B_R(x_0) \subset \Omega$. We take $\lambda>0$ large such that $\| f \|_{\lno} = \lambda |B_R(x_0)|^{1/N} > (\mathcal{S}_1 \inf_{s\geq 0} h(s))^{-1}$.

Suppose that there exists a solution $u\in BV(\Omega)$ to~\eqref{eq:PbMain}. Integrating in the ball $B_R(x_0)$ the expression~\eqref{eq:def_dist} and applying the Green formula~\eqref{eq:Green}, we obtain
\[
-\int_{\partial B_R(x_0)} [z,\nu] \ \dH = \int_{B_R(x_0)} g(u)|Du| + \int_{B_R(x_0)} h(u)f \geq \inf_{s\geq 0} h(s) \int_{B_R(x_0)} f = \inf_{s\geq 0} h(s) \lambda |B_R(x_0)|.
\]
Using that $\|[z,\nu]\|_{L^\infty(\partial B_R(x_0))} \leq 1$, it follows
\[
\inf_{s\geq 0} h(s) \|f\|_{\lno} = \inf_{s\geq 0} h(s) \lambda |B_R(x_0)|^{1/N} \leq \frac{\mathcal{H}^{N-1}(\partial B_R(x_0))}{|B_R(x_0)|^{(N-1)/N}} = N\omega_N^{1/N} = \mathcal{S}_1^{-1},
\]
which contradicts that $\| f \|_{\lno} > (\mathcal{S}_1 \inf_{s\geq 0} h(s))^{-1}$.
\end{example}

In Theorem~\ref{th:Exist}, we have shown the existence of a bounded solution. However, this qualitative property is not inherent to the structure of the problem, as unbounded solutions may also exist. Indeed, this occurs at least when $g, h\equiv 1$ and the datum is small and sufficiently integrable, as shown in~\cite{AbDaSe}. In what follows, we present a problem with singular terms $g$ and $h$ that admits an unbounded solution.

\begin{example}[Existence of unbounded solutions]
\label{ex:1}
We look for explicit solutions of problem ~\eqref{eq:PbMain} in $\Omega=B_R(0)$. We consider the model  case
\begin{equation}
\label{eq:PbExample1}
	\begin{cases}
		-\Delta_1 u = \lambda\frac{|Du|}{u^\theta} + \frac{f(x)}{u^\gamma} & \text{in } B_R(0),\\
		u=0 & \text{on } \partial B_R(0),
	\end{cases}
\end{equation}
where $\lambda, \theta, \gamma>0$. Assume that $u$ is a radially decreasing solution, i.e., that $u(x)=\rho(|x|)$ with $\rho'(s)<0$ for all $s\in(0,R)$. This implies that $z(x)=\frac{Du}{|Du|} = -\frac{x}{|x|}$ and that $-\diver z = \frac{N-1}{|x|}$. Observe that in this case the boundary condition is always verified, as $[z,\nu](x)=-1$ for all $x\in\partial B_R(0)$. Then, $u$ will be a solution of~\eqref{eq:PbExample1} if it belongs to $BV(\Omega)$ and satisfies the equation
\begin{equation}
\label{eq:PbExample2}
\frac{N-1}{|x|} = \lambda\frac{|\rho'(|x|)|}{\rho(|x|)^\theta} + \frac{f(x)}{\rho(|x|)^\gamma}.
\end{equation}
Assume now that $\theta=\gamma=1$, and consider $f(x)= \frac{N-1}{|x|} \in L^{N,\infty}(B_R(0))$. Then, it can be easily checked that
\[
u(x)=\rho(|x|) = 1+\left(\frac{R}{|x|}\right)^\frac{N-1}{\lambda}
\]
solves the equation~\eqref{eq:PbExample2} and
belongs to $BV(B_R(0))$ if $\lambda>1$. Thus, $u$ is an unbounded solution to~\eqref{eq:PbExample1} which differs from the one provided by Theorem~\ref{th:Exist}. Finally, observe that this explicit solution is positive in $\partial B_R(0)$, in agreement with Theorem~\ref{th:Trace} and Corollary~\ref{cor:Trace}.
\end{example}

Finally, we show that the non-integrability of $g$ in Theorem~\ref{th:Trace} is a necessary assumption. In Section~\ref{sec:trivial_solutions}, we provide several problems that admit $u\equiv 0$ as a solution when $g$ is integrable (that is, when it satisfies~\eqref{eq:hyp_g_int}) and $h(0)<+\infty$. With the example below, we show that non-trivial zero boundary trace solutions also exist when $g$ is integrable near the origin and $h$ is singular.

\begin{example}[Existence of zero trace solutions when $g$ is integrable]\label{ex:2}

As in Example~\ref{ex:1}, we look for a radially decreasing solution to~\eqref{eq:PbExample1}. We now consider $\theta<1$; note that $g(s)=s^{-\theta}$ satisfies~\eqref{eq:hyp_g_int}. In this case, equation~\eqref{eq:PbExample2} can be rewritten as
\begin{equation}
\label{eq:PbExample3}
\frac{N-1}{|x|} = \frac{\lambda}{1-\theta} \left| \left(\rho(|x|)^{1-\theta} \right)' \right| + \frac{f(x)}{\rho(|x|)^\gamma}.
\end{equation}
By considering the function
\[
f(x)= \big(R-|x|\big)^{\frac{\gamma}{1-\theta}} \left[ \frac{N-1}{|x|} - \frac{\lambda}{1-\theta} \right] \in L^{N,\infty}(B_R(0)),
\]
which is positive in $B_R(0)$ when $\lambda\leq (1-\theta)\frac{N-1}{R}$, one can easily verify that
\[
u(x)=\rho(|x|) = (R-|x|)^\frac{1}{1-\theta}
\]
solves~\eqref{eq:PbExample3}. Consequently, $u$ is a bounded solution to~\eqref{eq:PbExample1} such that $u(x)=0$ for all $x\in \partial\Omega$.
\end{example}

\section{Triviality and non-triviality when the zero-order term is non-singular}\label{sec:trivial_solutions}

As emphasized throughout this paper, the local integrability of $h(u)f$ in Definition~\ref{defsol} ensures that all solutions to~\eqref{eq:PbMain} are strictly positive in $\Omega$ provided $h(0)=+\infty$. In the present section, we investigate whether solutions are trivial (i.e., identically zero) or not when $h(0)<+\infty$. As we show, this depends entirely on the size of $f$: when $f$ is small, problem~\eqref{eq:PbMain} admits only trivial  bounded  solutions, whereas when $f$ is large, any solution is non-trivial. Between these thresholds, we prove the existence of a ``transition regime'', where trivial and non-trivial solutions can coexist.

For the sake of clarity, we first state the results of this section using the $W^{-1,\infty}(\Omega)$-norm  for data functions $f$  and then, when possible, we reformulate them in terms of the $L^{N,\infty}(\Omega)$-norm. Recall that $W^{-1,\infty}(\Omega)$ is defined as the dual space of $W_0^{1,1}(\Omega)$ and is endowed with the norm
\begin{equation*}
\|f\|_{W^{-1,\infty}(\Omega)} = \sup \left\{ \int_\Omega f \phi \;:\; \phi \in W_0^{1,1}(\Omega),\ \int_\Omega |\nabla \phi| \le 1 \right\},\ 
\end{equation*}
  for any function  $f\in W^{-1,\infty}(\Omega)$. 
It should be noted that the space $L^{N,\infty}(\Omega)$ is strictly contained in $W^{-1,\infty}(\Omega)$. 

To show the sharpness of our results, we provide several examples that require computing the norms of some functions. Specifically, we recall that $\|1\|_{W^{-1,\infty}(B_1)} = \frac{1}{N}$ and $\|1\|_{L^{N,\infty}(B_1)} = \omega_N^{1/N}$, as well as $\big\| |x|^{-1} \big\|_{W^{-1,\infty}(B_1)} = \frac{1}{N-1}$ and $\big\| |x|^{-1} \big\|_{L^{N,\infty}(B_1)} = \omega_N^{1/N}$, where $B_1$ denotes the unit ball and $\omega_N$ its volume.

The following general theorem provides two distinct thresholds for the triviality and non-triviality of solutions depending on the norm in $W^{-1,\infty}(\Omega)$ of $f$. We stress that the results contained in this section remain valid if $g\equiv 0$ (see Remark~\ref{rem:Dual_g0} below), thus extending and completing previous related results in~\cite{CiTr, DecGOP, Kaw, MeSeTr}.

\begin{theorem}\label{th:Dual} 
Let $0<f\in W^{-1,\infty}(\Omega)$. Assume that $g$ satisfies the integrability condition~\eqref{eq:hyp_g_int} and suppose that $h(0)<+\infty$. Let $u\in\bvo$ be a solution to~\eqref{eq:PbMain}. Then the following holds:
\begin{enumerate}[i)]
    \item if $\|f\|_{W^{-1,\infty}(\Omega)} < \frac{1}{\sup_{s\geq 0} h(s)}$ and $u\in \lio$, then $u\equiv 0$ a.e. in $\Omega$,
    \item if $\|f\|_{W^{-1,\infty}(\Omega)} > \frac{1}{h(0)}$, then $u$ is non-trivial.
\end{enumerate}
Furthermore, if $\|f\|_{W^{-1,\infty}(\Omega)} \leq \frac{1}{h(0)}$, then $u\equiv 0$ is a solution to~\eqref{eq:PbMain}.
\end{theorem}

\begin{remark}
The request $u \in L^\infty(\Omega)$ in the first item of Theorem~\ref{th:Dual} is actually necessary. In fact, unbounded solutions exist when the datum is small and sufficiently integrable, as shown in \cite{AbDaSe} for the case $g, h\equiv 1$. Moreover, the assumption that $g$ satisfies the integrability condition \eqref{eq:hyp_g_int} is not restrictive in the context of Theorem \ref{th:Dual}. Indeed, since any solution in the sense of Definition~\ref{defsol} satisfies $\Gamma(u)\in BV_{\rm loc}(\Omega)$, if~\eqref{eq:hyp_g_int} fails to hold, then any solution must be positive in $\Omega$ because $\Gamma(0)=-\infty$ (in fact \eqref{eq:prov19} holds).
\end{remark}

\begin{proof}[Proof of Theorem~\ref{th:Dual}]

We first prove item $i)$. Following the ideas of~\cite[Theorem~2]{DaSe}, let $\lambda > 1$, and let us take $\psi(u) = e^{\lambda \Gamma(u)} - e^{\lambda \Gamma(0)}$ as test function in~\eqref{eq:PbMain}. Note that this function belongs to $BV(\Omega) \cap L^\infty(\Omega)$ by the chain rule, since $\Gamma(u) \in \bvo \cap \lio$. Indeed, $\Gamma(u)$ is bounded because $u\in\lio$ and~\eqref{eq:hyp_g_int} holds, and $D\Gamma(u) \in \mathcal{M}(\Omega)^N$ holds due to~\eqref{eq:def_dist}. Therefore, applying Green's identity~\eqref{eq:Green}, one has
\begin{equation*}
    \int_\Omega (z, D\psi(u)) - \int_{\partial\Omega} \psi(u) [z, \nu] \ \dH = \int_\Omega g(u) \psi(u) |Du| + \int_\Omega h(u) f \psi(u).
\end{equation*}
Since $s\mapsto \psi(s)$ is an increasing function, Lemma~\ref{lem:Composition} implies that $(z,D\psi(u))=|D\psi(u)|$. Moreover, from the boundary condition \eqref{eq:def_bord} and the fact that $\psi(0)=0$, it follows that $-\psi(u) [z, \nu] = \psi(u)$ on $\partial\Omega$. Thus, we obtain
\begin{equation}
\label{eq:Pf_Dual}
\int_\Omega |D\psi(u)| + \int_{\partial\Omega} \psi(u) \ \dH = \int_\Omega g(u) (e^{\lambda \Gamma(u)} - e^{\lambda \Gamma(0)}) |Du| + \int_\Omega h(u) f \psi(u).
\end{equation}
By the chain rule (Lemma~\ref{lem:ChainRule}), we have $g(u) e^{\lambda \Gamma(u)}|Du| = \frac{1}{\lambda} |D\psi(u)|$. Taking this into account and dropping a term in~\eqref{eq:Pf_Dual}, it follows
\begin{equation*}
    \frac{\lambda - 1}{\lambda} \int_\Omega |D\psi(u)| + \int_{\partial\Omega} \psi(u) \ \dH \le \int_\Omega h(u) f \psi (u).
\end{equation*}
Finally, observing that
\begin{equation*}
    \int_\Omega h(u) f \psi(u) \le \sup_{s\geq 0} h(s) \|f\|_{W^{-1,\infty}(\Omega)} \|\psi(u)\|_{BV(\Omega)},
\end{equation*}
one obtains
\begin{equation*}
    \left( \frac{\lambda - 1}{\lambda} -  \sup_{s\geq 0} h(s) \|f\|_{W^{-1,\infty}(\Omega)} \right) \|\psi(u)\|_{BV(\Omega)} \le 0.
\end{equation*}
Since we have assumed that $\|f\|_{W^{-1,\infty}(\Omega)} < \left(\sup_{s\geq 0}h(s) \right)^{-1}$, we can choose $\lambda > 1$ sufficiently large so that the coefficient in parentheses is strictly positive. This implies that $\|\psi(u)\|_{BV(\Omega)} = 0$ and, since $\psi(s)$ only vanishes at $s=0$, we conclude that $u \equiv 0$ a.e. in $\Omega$.

We now show item \textit{ii)}. We proceed by contradiction, assuming that $u \equiv 0$ a.e. in $\Omega$. 
In this case, $u$ satisfies (see~\eqref{eq:def_dist})
\begin{equation}\label{eq:Pf_Dual_2}
    -\operatorname{div} z = h(0)f \text{ as measures in } \Omega.
\end{equation}
Then, for any $\phi \in W_0^{1,1}(\Omega)$, it follows from~\eqref{eq:Pf_Dual_2} that
\begin{equation*}
    \int_\Omega h(0)f\phi   = \int_\Omega z \cdot \nabla \phi.
\end{equation*}
In particular, since  $|z|\le 1$ in $\Omega$, one obtains
\begin{equation*}
    \|h(0)f\|_{W^{-1,\infty}(\Omega)} \le 1, 
\end{equation*}
which contradicts our assumption. Therefore, $u$ cannot be identically zero a.e. in $\Omega$.

Finally, we prove that $u\equiv 0$ is a solution to~\eqref{eq:PbMain} if $\|f\|_{W^{-1,\infty}(\Omega)} \leq \frac{1}{h(0)}$. Since the boundary condition~\eqref{eq:def_bord} is satisfied and~\eqref{eq:def_pair} holds for any vector field $z$, we only need to find a vector field $z$ that verifies~\eqref{eq:def_dist} for $u\equiv 0$. Because $f\in W^{-1,\infty}(\Omega)$, it can be represented as $f=-\operatorname{div}w$, where $w\in\lio^N$ satisfies $\|w\|_{\lio^N} = \|f\|_{W^{-1,\infty}(\Omega)}$. Defining $z:= h(0)w$, it follows that $\|z\|_{\lio^N} = h(0) \|f\|_{W^{-1,\infty}(\Omega)} \leq 1$ and
\[
-\operatorname{div} z = h(0)f.
\]
This shows that $u\equiv 0$ is a solution to~\eqref{eq:PbMain}.
\end{proof}
\begin{remark}
\label{rem:Dual_g0}
The results of this section also cover the case $g\equiv 0$, although item \textit{i)} then requires a different test function: with our normalization one has $\Gamma\equiv 0$, so that $\psi\equiv 0$ and the above argument becomes vacuous. In this case, however, one may simply take $\psi(u)=u\in\bvo\cap\lio$ as test function in~\eqref{eq:weak_form}. Since the gradient term is absent, reasoning exactly as above yields
\[
\int_\Omega |Du| + \int_{\partial\Omega} u \ \dH = \int_\Omega h(u) f u \leq \sup_{s\geq 0} h(s) \|f\|_{W^{-1,\infty}(\Omega)} \|u\|_{BV(\Omega)},
\]
whence $\left( 1 - \sup_{s\geq 0} h(s) \|f\|_{W^{-1,\infty}(\Omega)} \right) \|u\|_{BV(\Omega)} \leq 0$, and $u\equiv 0$ follows under the same smallness assumption. Item \textit{ii)} and the last statement of Theorem~\ref{th:Dual} do not involve $\psi$, so they hold verbatim.
\end{remark}

The different cases in Theorem~\ref{th:Dual}, which depend on the value of $\|f\|_{W^{-1,\infty}(\Omega)}$, are illustrated in Figure~\ref{fig:dual}. We identify a triviality range when $\|f\|_{W^{-1,\infty}(\Omega)}$ is strictly smaller than $\frac{1}{\sup_{s\geq 0} h(s)}$, below which $u\equiv 0$ is the unique bounded solution, and a non-triviality threshold when $\|f\|_{W^{-1,\infty}(\Omega)}$ is strictly larger than $\frac{1}{h(0)}$, above which any solution is non-trivial.

\begin{center}
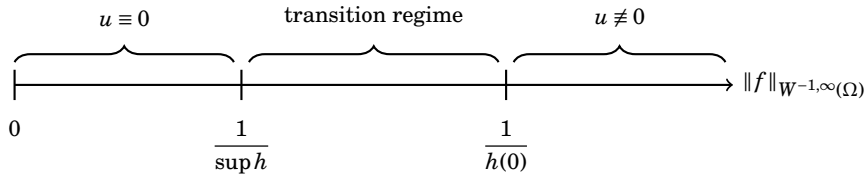

	\begin{tikzpicture}[
		scale=1,
		every node/.style={font=\small},
		brace/.style={decorate, decoration={brace, amplitude=6pt}, thick}
		]
		
			\draw[->, thick] (0,0) -- (9.5,0) node[right] {$\|f\|_{W^{-1,\infty}(\Omega)}$};

		\draw[thick] (0,0.2) -- (0,-0.2) node[below=0.15cm] {$0$};
		\draw[thick] (3,0.2) -- (3,-0.2) node[below=0.15cm] {$\displaystyle\frac{1}{\sup h}$};
		\draw[thick] (6.5,0.2) -- (6.5,-0.2) node[below=0.15cm] {$\displaystyle\frac{1}{h(0)}$};
		
		\draw[brace] (0,0.3) -- (2.9,0.3) node[midway, above=0.35cm, align=center] {$u\equiv 0$};
		\draw[brace] (3.1,0.3) -- (6.4,0.3) node[midway, above=0.35cm, align=center] {transition regime};
		\draw[brace] (6.6,0.3) -- (9.4,0.3) node[midway, above=0.35cm, align=center] {$u \not \equiv 0$};
	\end{tikzpicture}	
	\captionof{figure}{Different regimes in terms of $\|f\|_{W^{-1,\infty}(\Omega)}$ according to Theorem~\ref{th:Dual}.}
	\label{fig:dual}
\end{center}

\medskip 

Between the triviality and non-triviality thresholds, there is a gap that we refer to as the ``transition regime''. When the norm of $f$ lies in this interval, $u\equiv 0$ is always a solution. Nevertheless, non-trivial bounded solutions may also exist, as shown in the following example.

\begin{example}\label{ex:transition}
Let $\Omega=B_1$ be the unit ball, let $h$ be such that $0<h(0)<\sup_{s\geq 0}h(s)$ and let $f\equiv \kappa N$ with
\[ 
\kappa\in\left(\frac{1}{\sup_{s\geq 0} h(s)},\ \frac{1}{h(0)}\right).
\]
In particular, since $f$ is constant, it holds $\|f\|_{W^{-1,\infty}(B_1)}=\frac{f}{N}=\kappa$. Then problem~\eqref{eq:PbMain} admits as solutions:
\begin{itemize}
    \item[(i)] \emph{The trivial solution}, which is guaranteed by Theorem~\ref{th:Dual}. In this specific example, the vector field $z$ can be found explicitly: note that $z(x)=-h(0)\kappa\,x$ is admissible because $\|z\|_{L^\infty(B_1)^N}=h(0)\kappa\le 1$ and solves $-\operatorname{div}z=h(0)f=h(0)\kappa N$.
    
    \item[(ii)] \emph{A non-trivial solution.}  We look for a constant solution $u\equiv\beta>0$. Taking $z(x)=-x$, which is admissible since $\|z\|_{L^\infty(B_1)^N}=1$, the boundary condition $[z,\nu]=-1$ on $\partial B_1$ is satisfied and the equation $-\operatorname{div}z=h(u)f$ becomes $N=h(\beta)\kappa N$, that is, $h(\beta)=\kappa^{-1}$. Therefore, any positive constant $\beta$ satisfying $h(\beta) = \kappa^{-1}$ is a solution to~\eqref{eq:PbMain}. Since $h(0) \leq \kappa^{-1} < \sup_{s\geq 0}h(s)$, the existence of such a $\beta>0$ follows from the intermediate value theorem.
\end{itemize}
Finally, we stress that if the supremum $\sup_{s\geq 0} h(s)$ is attained, the same procedure yields a non-trivial solution for $\kappa =\frac{1}{\sup_{s\geq 0} h(s)}$.
\end{example}

The above Theorem~\ref{th:Dual} can be suitably rephrased in terms of the $L^{N,\infty}(\Omega)$-norm simply by taking into account the relationship between both norms. On the one hand, by applying the Sobolev embedding $W_0^{1,1}(\Omega) \hookrightarrow L^{\frac{N}{N-1},1}(\Omega)$ directly in the definition of $\|f\|_{W^{-1,\infty}(\Omega)}$, one obtains
\[
\|f\|_{W^{-1,\infty}(\Omega)} \leq \tilde{\mathcal{S}}_1\|f\|_{L^{N,\infty}(\Omega)} = \frac{1}{(N-1) \omega_N^{1/N}}\|f\|_{L^{N,\infty}(\Omega)},\ \forall f\in L^{N,\infty}(\Omega). 
\]
On the other hand, when $\Omega=B_R$ is a ball, the spaces $L^{N,\infty}(\Omega)$ and $W^{-1,\infty}(\Omega)$ coincide within the class of non-negative, radially  non-increasing functions. In this specific case, one has
\[
\|f\|_{W^{-1,\infty}(\Omega)} \geq \frac{N-1}{N}\tilde{\mathcal{S}}_1\|f\|_{L^{N,\infty}(\Omega)} = \frac{1}{N \omega_N^{1/N}}\|f\|_{L^{N,\infty}(\Omega)},\ \forall 0\leq f\in L^{N,\infty}(B_R) \text{ radially decreasing}.  
\]
Taking these inequalities into account, one obtains the following corollary from Theorem~\ref{th:Dual}.

\begin{corollary}
\label{cor:weakLp}
Let $0<f\in L^{N,\infty}(\Omega)$. Assume that $g$ satisfies the integrability condition~\eqref{eq:hyp_g_int} and suppose that $h(0)<+\infty$. Let $u\in\bvo$ be a solution to~\eqref{eq:PbMain}. Then the following holds:
\begin{enumerate}[i)]
    \item if $\|f\|_{L^{N,\infty}(\Omega)} < \frac{1}{\tilde{ \mathcal{S}}_1\sup_{s\geq 0} h(s)}$ and $u\in \lio$, then $u\equiv 0$ a.e. in $\Omega$,
    \item if $\|f\|_{L^{N,\infty}(\Omega)} > \frac{N}{N-1}\frac{1}{\tilde{ \mathcal{S}}_1 h(0)}$, $\Omega$ is a ball and $f$ is radially decreasing, then $u$ is non-trivial.
\end{enumerate}
Furthermore, if $\|f\|_{L^{N,\infty}(\Omega)} \leq \frac{1}{\tilde{ \mathcal{S}}_1 h(0)}$, then $u\equiv 0$ is a solution to~\eqref{eq:PbMain}.
\end{corollary}

\begin{remark}
If $\sup_{s\geq 0} h(s) = h(\infty)$, the existence threshold given in Theorem \ref{th:Exist} coincides with the triviality range of Corollary~\ref{cor:weakLp}. In this scenario, the bounded solution provided by Theorem \ref{th:Exist} is identically zero. Nevertheless, the above corollary shows that when the interval $\left(\frac{N}{N-1}\frac{1}{h(0)}, \frac{1}{h(\infty)}\right)$ is non-empty, one can find data and domains for which the solution given by Theorem \ref{th:Exist} is non-trivial.
\end{remark}

The different thresholds appearing in Corollary~\ref{cor:weakLp} are illustrated in Figure~\ref{fig:lorentz}. Note that the range for non-triviality applies only when $\Omega$ is a ball and $f$ is radially decreasing. As before (see Figure~\ref{fig:dual}), below the triviality threshold, $u\equiv 0$ is the unique bounded solution, whereas within the non-triviality range, any solution is non-trivial. Furthermore, in the ``transition regime'' $u\equiv 0$ is always a solution, but it can coexist with non-trivial bounded solutions.

\begin{center}
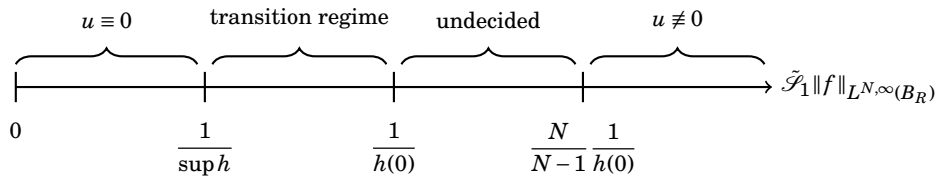

    \begin{tikzpicture}[
        scale=1,
        every node/.style={font=\small},
        brace/.style={decorate, decoration={brace, amplitude=6pt}, thick}
        ]
        
        \draw[->, thick] (0,0) -- (10,0) node[right] {$\tilde{\mathcal{S}}_1\|f\|_{L^{N,\infty}(B_R)}$};

        \draw[thick] (0,0.2) -- (0,-0.2) node[below=0.15cm] {$0$};
        \draw[thick] (2.5,0.2) -- (2.5,-0.2) node[below=0.15cm] {$\displaystyle\frac{1}{\sup h}$};
        \draw[thick] (5,0.2) -- (5,-0.2) node[below=0.15cm] {$\displaystyle\frac{1}{h(0)}$};
        \draw[thick] (7.5,0.2) -- (7.5,-0.2) node[below=0.15cm] {$\dfrac{N}{N-1}\dfrac{1}{h(0)}$};
        
        \draw[brace] (0,0.3) -- (2.4,0.3) node[midway, above=0.35cm, align=center] {$u\equiv 0$};
        \draw[brace] (2.6,0.3) -- (4.9,0.3) node[midway, above=0.35cm, align=center] {transition regime};
        \draw[brace] (5.1,0.3) -- (7.4,0.3) node[midway, above=0.35cm, align=center] { undecided};
        \draw[brace] (7.6,0.3) -- (9.9,0.3) node[midway, above=0.35cm, align=center] { $u\not\equiv 0$};
    \end{tikzpicture}	
        \captionof{figure}{Different regimes in terms of $\tilde{\mathcal{S}}_1\|f\|_{L^{N,\infty}(\Omega)}$ according to Corollary~\ref{cor:weakLp} when $\Omega$ is a ball and $f$ is  radially decreasing.}
    \label{fig:lorentz}
\end{center}

The main difference in this setting is the appearance of an ``undecided'' region. In this region, any scenario may occur, strongly depending on the value of $\|f\|_{W^{-1,\infty}(\Omega)}$. The main difference with respect to the ``transition regime'' is that $u\equiv 0$ is not always a solution, although for some data $u\equiv 0$ remains the unique bounded solution. The following example illustrates these scenarios for different data $f$ sharing the same value of $\tilde{\mathcal{S}}_1\|f\|_{L^{N,\infty}(\Omega)}$.

\begin{example}
Let $\Omega=B_1$ be the unit ball, let $h$ be such that $h(0)>0$ and let
\[
\lambda\in\left(\frac{1}{h(0)},\ \frac{N}{N-1}\frac{1}{h(0)}\right).
\]
In the following, we present two functions $f_1$ and $f_2$ such that $\tilde{\mathcal{S}}_1\|f_1\|_{L^{N,\infty}(B_1)} = \tilde{\mathcal{S}}_1\|f_2\|_{L^{N,\infty}(B_1)} = \lambda$, where problem~\eqref{eq:PbMain} admits the trivial solution for $f_1$, whereas it admits only non-trivial solutions for $f_2$.
\begin{itemize}
    \item[(i)] Let $f_1 \equiv \lambda(N-1)$. In this case, $\tilde{\mathcal{S}}_1\|f_1\|_{L^{N,\infty}(B_1)} = \lambda$ and $\|f_1\|_{W^{-1,\infty}(B_1)} = \lambda\frac{N-1}{N} < \frac{1}{h(0)}$. By Theorem~\ref{th:Dual}, $u\equiv 0$ is a solution to~\eqref{eq:PbMain}. Indeed, an admissible vector field is given by $z(x)=-\frac{\lambda(N-1)h(0)}{N}x$.
    
    \item[(ii)] Let $f_2(x) = \lambda \frac{N-1}{|x|}$. Here, $\tilde{\mathcal{S}}_1\|f_2\|_{L^{N,\infty}(B_1)} = \lambda$ and $\|f_2\|_{W^{-1,\infty}(B_1)} = \lambda > \frac{1}{h(0)}$. By Theorem~\ref{th:Dual}, any bounded solution to~\eqref{eq:PbMain} is non-trivial. In fact, such a solution exists if there is some $\beta>0$ such that $h(\beta) = \lambda^{-1}$: it suffices to consider $z(x)=-\frac{x}{|x|}$ and $u\equiv \beta$.
\end{itemize}
\end{example}

Next, we provide an example dealing with the ``transition regime'' in the $L^{N,\infty}(\Omega)$ setting. In this scenario, Example~\ref{ex:transition} does not apply, as the quantity $\tilde{\mathcal{S}}_1\|f\|_{L^{N,\infty}(\Omega)}$ falls outside the transition range (see Figure~\ref{fig:lorentz}). The example is as follows.

\begin{example}\label{ex6}\label{ex:transition_Lorentz}
Let $\Omega=B_1$ be the unit ball, let $h$ be such that $h(0)<\sup_{s\geq 0}h(s)$ and let $f(x)=\lambda\frac{N-1}{|x|}$ with
\[ 
\lambda\in\left(\frac{1}{\sup_{s\geq 0} h(s)},\ \frac{1}{h(0)}\right).
\]
Note that $\tilde{\mathcal{S}}_1\|f\|_{L^{N,\infty}(B_1)} = \frac{1}{(N-1)\omega_N^{1/N}} \|f\|_{L^{N,\infty}(B_1)} =\lambda$. Then problem~\eqref{eq:PbMain} admits as solutions:
\begin{itemize}
    \item[(i)] \emph{The trivial solution}, which is guaranteed by Corollary~\ref{cor:weakLp}. Indeed, an admissible vector field $z$ can be found explicitly: $z(x)=-h(0)\lambda \frac{x}{|x|}$ satisfies $\|z\|_{L^\infty(B_1)^N}=h(0)\lambda\leq 1$ and solves $-\operatorname{div}z=h(0)f$.

    \item[(ii)] \emph{A non-trivial solution.} We look for a constant solution $u\equiv\beta>0$. Taking $z(x)=-\frac{x}{|x|}$, which is admissible since $\|z\|_{L^\infty(B_1)^N}=1$, the boundary condition $[z,\nu]=-1$ on $\partial B_1$ is satisfied and the equation $-\operatorname{div}z=h(u)f$ becomes $\frac{N-1}{|x|}=h(\beta)\lambda \frac{N-1}{|x|}$, that is, $h(\beta)=\lambda^{-1}$. Therefore, any positive constant $\beta$ satisfying $h(\beta) = \lambda^{-1}$ is a solution to~\eqref{eq:PbMain}. Since $h(0) \leq \lambda^{-1} < \sup_{s\geq 0}h(s)$, the existence of such a $\beta>0$ follows from the intermediate value theorem.
\end{itemize}
Finally, we stress that if the supremum $\sup_{s\geq 0} h(s)$ is attained, the same construction yields a non-trivial solution for $\lambda =\frac{1}{\sup_{s\geq 0} h(s)}$.
\end{example}

The solutions in the examples above are constant. We discuss the existence of such solutions in the remark below.

\begin{remark}
The non-trivial solutions exhibited in Examples~\ref{ex:transition}--\ref{ex:transition_Lorentz} are constant, but this is by no means the typical situation: it is a consequence of a geometric property of the ball, and it fails on general domains. To make this precise, recall that the \textit{Cheeger constant} of a bounded Lipschitz domain $\Omega$ is
\begin{equation*}
\lambda_\Omega := \inf \left\{ \frac{P(F)}{|F|} \;:\; F\subseteq \Omega,\ |F|>0,\ P(F)<+\infty \right\},
\end{equation*}
where $P(F)$ denotes the perimeter in $\R^N$ of (the extension by zero of) $F$; a minimizer is called a \textit{Cheeger set} of $\Omega$, and $\Omega$ is said to be \textit{self-Cheeger} if it is a Cheeger set of itself, i.e. if $\lambda_\Omega = \frac{P(\Omega)}{|\Omega|}$.

Let now $\Omega$ be any bounded Lipschitz domain, let $f\equiv c>0$, and assume that $u\equiv\beta>0$ is a solution to~\eqref{eq:PbMain}. Since $|Du|=0$, the term involving $g$ vanishes and $-\operatorname{div}z=h(\beta)c=:\eta$ is a constant, while $u>0$ on $\partial\Omega$ forces $[z,\nu]=-1$ by~\eqref{eq:def_bord}. Testing the Green formula~\eqref{eq:Green} with $v\equiv 1$ gives
\[
\eta|\Omega|=\int_\Omega v\,(-\operatorname{div}z)=-\int_{\partial\Omega}[z,\nu]\ \dH=P(\Omega),
\]
whereas testing it with $v=\chi_F$, where $F\subseteq\Omega$ has finite perimeter and $|F|>0$, and using $|(z,Dv)|\le \|z\|_{L^\infty(\Omega)^N}|Dv|\le |Dv|$ together with $[z,\nu]=-1$ on $\partial\Omega$, gives
\[
\eta|F|\;\le\;|D\chi_F|(\Omega)+\int_{\partial\Omega}\chi_F\ \dH = P(F),
\]
the right-hand side being the total variation in $\R^N$ of the extension of $\chi_F$ by zero. Taking the infimum over such $F$ we obtain $\eta\le \lambda_\Omega\le\frac{P(\Omega)}{|\Omega|}=\eta$, so that necessarily
\[
h(\beta)\,c=\lambda_\Omega=\frac{P(\Omega)}{|\Omega|},
\]
i.e. $\Omega$ must be self-Cheeger. If this is not the case (for instance, if $\Omega$ is a rectangle), then \emph{no} positive constant solution exists in $\Omega$ for a constant datum, whatever $h$ and $c$ are, and hence every
non-trivial solution is non-constant.  This already happens in the simplest situation: if $g\equiv0$, $h(s)=s^{-\gamma}$ and $f\equiv1$, the solution to~\eqref{eq:PbMain} on a convex domain is $u=(-H_\Omega)^{-1/\gamma}$, where $H_\Omega$ denotes the variational mean curvature of $\Omega$, which is constant on the Cheeger set of $\Omega$ and blows up at the corners, so that
the solution $u$ vanishes there (see~\cite[Example~2]{GOP}). Non-constant solutions also appear in Examples~\ref{ex:1} and~\ref{ex:2} above.

On the other hand, the ball plays no special role beyond being calibrable, and the coexistence phenomenon of Example~\ref{ex:transition} carries over to any calibrable domain, up to replacing the constant $N$ by the Cheeger constant $\lambda_\Omega$. Precisely, assume that $\Omega$ admits a \textit{calibration}, i.e. a vector field $z_\Omega\in\DM(\Omega)$ such that
\[
\|z_\Omega\|_{L^\infty(\Omega)^N}\le1, \qquad -\operatorname{div}z_\Omega=\lambda_\Omega \ \text{ in } \Omega, \qquad [z_\Omega,\nu]=-1 \ \text{ on } \partial\Omega,
\]
which in particular forces $\Omega$ to be self-Cheeger. Let $h$ be such that $0<h(0)<\sup_{s\geq0}h(s)$ and let
\[
\kappa\in\left(\frac{1}{\sup_{s\geq 0}h(s)},\ \frac{1}{h(0)}\right).
\]
Then the datum $f\equiv \lambda_\Omega\kappa$ satisfies $\|f\|_{W^{-1,\infty}(\Omega)}=\frac{\lambda_\Omega \kappa}{\lambda_\Omega}=\kappa$, since $\|c\|_{W^{-1,\infty}(\Omega)}=c\,\lambda_\Omega^{-1}$ for every constant $c>0$, and both
\begin{itemize}
    \item[(i)] $u\equiv 0$, realized by $z=h(0)\kappa\, z_\Omega$, which is admissible as $\|z\|_{L^\infty(\Omega)^N}\le h(0)\kappa\le1$, and
    \item[(ii)] $u\equiv\beta$ with $h(\beta)=\kappa^{-1}$, realized by $z=z_\Omega$,
\end{itemize}
are solutions to~\eqref{eq:PbMain}. For $\Omega=B_R$ one has $\lambda_{B_R}=\frac{N}{R}$ and $z_\Omega(x)=-\frac{x}{R}$, which gives back Example~\ref{ex:transition} for $R=1$.
\end{remark}

\appendix

\section{Existence of bounded solution for a problem with no singularities}
\label{sec:Appendix}

In this appendix, we prove the existence of a solution $u\in \bvo\cap \lio$ to the approximated problems~\eqref{eq:PbApprox}. Although this can be inferred from~\cite{GOP}, we prefer to provide a proof here to make this work self-contained. Our demonstration, which follows the approach of~\cite{GOP}, considerably simplifies some parts, mainly because we deal with non-singular coefficients.

The problem we study here is
\begin{equation}
	\label{eq:PbAppendix}
	\begin{cases}
		-\Delta_1 u = g(u)|Du| + h(u)f & \text{in }\Omega,\\
		u=0 & \text{on } \partial\Omega,
	\end{cases}
\end{equation}
where $g,h\in C(\R)$ with $g>0$ and $f\in L^{N,\infty}(\Omega)$.  In~\cite{GOP}, the authors are only interested in studying the case where $f\geq 0$, obtaining the existence of a non-negative solution provided that $h\geq 0$. For completeness, we also address the case where $f$ and $h$ may change sign.

Since the zero-order term of~\eqref{eq:PbAppendix} can change sign, Definition~\ref{defsol} must be reformulated for sign-changing solutions, as the pairing $(z,Du)$ is not well defined in that setting. Nevertheless, we circumvent this problem by restricting ourselves to bounded solutions, where $(z,Du)$ is well defined. The notion of solution for problem~\eqref{eq:PbAppendix} is as follows.

\begin{defin}\label{defsolA}
Let $f\in L^1(\Omega)$. A function $u\in BV(\Omega)\cap L^{\infty}(\Omega)$ is a \textit{bounded solution} to problem~\eqref{eq:PbAppendix} if $D^j u=0$ and if there exists $z\in \mathcal{D}\mathcal{M}^\infty(\Omega)$ with $\|z\|_{L^\infty(\Omega)^N}\le 1$ such that
\begin{gather} 
    \label{eq:def_distA}
    -\operatorname{div}z = g(u)|Du| + h(u)f \text{ as measures in } \Omega,\\
    \label{eq:def_pairA}
    (z,Du)=|Du| \text{ as measures in } \Omega,\\
    \label{eq:def_bordA}
    |u(x)|+u(x)[z,\nu](x) =0 \text{ for  $\mathcal{H}^{N-1}$-a.e. } x \in \partial\Omega.		
\end{gather}
\end{defin}

\begin{remark}
Any solution in the sense of Definition~\ref{defsolA} satisfies the weak formulation
\begin{equation}
\label{eq:app_weak_formA}
\into (z, D\phi) - \intdo \phi [z,\nu] \ \dH = \into g(u) \phi |Du| + \into h(u) \phi f, \ \forall \phi \in \bvo\cap\lio.
\end{equation}
This follows from Green's identity~\eqref{eq:Green} applied to~\eqref{eq:def_distA} after taking $\phi\in\bvo\cap\lio$ as test function.
\end{remark}

We show the existence of a bounded solution to~\eqref{eq:PbAppendix} under elementary assumptions. First, we assume that $h$ remains bounded as $|s|$ diverges, that is,
\begin{equation}
\label{eq:ap_hyp_inf}
h(\infty) := \limsup_{s\to \pm \infty} |h(s)| <+ \infty.
\end{equation}
We also require a smallness condition in the norm of $f$, namely that $\|f\|_{L^{N,\infty}(\Omega)} < (\tilde{\mathcal{S}}_1 h(\infty))^{-1}$. Under these assumptions, there exists some constant $M>0$ such that
\begin{equation}
\label{eq:ap_M}
\tilde{\mathcal{S}}_1^{-1} - \|f\|_{L^{N,\infty}(\Omega)} \sup_{|s|> M} |h(s)|>0.
\end{equation}

The existence result is the following.

\begin{theorem}
\label{th:Appendix}
Let $f\in L^{N,\infty}(\Omega)$ and let $g,h\in C(\R)$ with $g> 0$. Assume that $h$ verifies~\eqref{eq:ap_hyp_inf}. If $\|f\|_{L^{N,\infty}(\Omega)} < (\tilde{\mathcal{S}}_1 h(\infty))^{-1}$, then~\eqref{eq:PbAppendix} has a solution $u\in BV(\Omega)\cap \lio$ in the sense of Definition \ref{defsolA} such that
\[
\|u\|_\lio \leq M,
\]
where $M>0$ is any constant satisfying~\eqref{eq:ap_M}. Furthermore, if both $f$ and $h$ are non-negative, then $u\geq 0$.
\end{theorem}

The strategy to prove this result is to approximate problem~\eqref{eq:PbAppendix} by a family of problems involving the $p$-Laplacian, and then let $p\to 1^+$. With this aim, for $1<p<2$ we consider the problems
\begin{equation}
	\label{eq:Pb_pLap}
	\begin{cases}
		-\Delta_p u_p = g_p(u_p)|\nabla u_p|^p + h(u_p)f & \text{in }\Omega,\\
		u_p=0 & \text{on } \partial\Omega,
	\end{cases}
\end{equation}
where $g_p$ is a suitable truncation of $g$. More precisely, we fix $M>0$ such that~\eqref{eq:ap_M} holds, and we define $g_p\colon \R \to \R$ as the continuous function
\begin{equation*}
g_p(s) := \begin{cases}
    g(s) & \text{ if }\ |s|\leq M,\\
    -\frac{g(\mathrm{sign}(s)M)}{p-1} (|s|-M-p+1) & \text{ if }\ M < |s| < M+p-1, \\
    0 & \text{ if }\ |s|> M+p-1.
\end{cases}
\end{equation*}

The existence of a solution $u_p\in \sob\cap\lio$ to~\eqref{eq:Pb_pLap} is guaranteed by~\cite[Theorem~2.1]{PoSe}. This existence result applies since all its structural assumptions hold in our framework, bearing in mind, in particular, that $\lim_{s\to\pm\infty} g_p(s)=0$, $h$ is bounded, and $f\in L^{N,\infty}(\Omega)\subset L^m(\Omega)$ for some $m\in (N/p, N)$. 

In the following, we prove, among other properties, that $u_p$ is bounded in $\bvo$. In the proof, we use the elementary inequality
\begin{equation}
\label{eq:ap_elemental}
e^{ps} -1 \leq C_p (e^s-1)^p,\ \forall s\geq 1, \ \forall p\in (1,2),
\end{equation}
where $C_p>0$ is such that $C_p\to 1^+$ as $p\to 1^+$ (indeed, one may take $C_p = (e^p-1)(e-1)^{-p}$).

\begin{lemma}
\label{lem:Appendix_BV}
Under the assumptions of Theorem~\ref{th:Appendix}, the sequence $u_p$ of solutions to~\eqref{eq:Pb_pLap} is bounded in $\bvo$. Furthermore, there exists some $u\in BV(\Omega)$ and some subsequence of $u_p$, not relabelled, such that
\begin{gather*}
    u_p\to u \text{ in } L^q(\Omega),\ \forall q\in[1,+\infty),\\
    \nabla u_p \rightharpoonup Du\ *\text{-weakly as measures in }\Omega.
\end{gather*}
\end{lemma}

\begin{proof}
Let $\eta \geq 1$. Taking $\left(e^{\eta p |u_p|}-1\right) \mathrm{sign}(u_p) \in \sob\cap\lio$ as test function in~\eqref{eq:Pb_pLap}, we deduce that
\begin{equation}
\label{eq:Pf_Append_1}
\eta p \into e^{\eta p |u_p|} |\nabla u_p|^p \leq \into g_p(u_p) \left(e^{\eta p |u_p|}-1 \right) |\nabla u_p|^p + \into |h(u_p)| \left(e^{\eta p |u_p|}-1 \right) |f|.
\end{equation}
With respect to the second integral in~\eqref{eq:Pf_Append_1}, recalling that $g_p(s)=0$ when $|s|\geq M+p-1$, we have
\begin{equation}
\label{eq:Pf_Append_2}
\into g_p(u_p) \left(e^{\eta p |u_p|}-1 \right) |\nabla u_p|^p = \int_{\{|u_p|\leq M+p-1\}} g_p(u_p) \left(e^{\eta p |u_p|}-1 \right) |\nabla u_p|^p \leq \max_{|s|\leq M} g(s) \into e^{\eta p |u_p|} |\nabla u_p|^p.
\end{equation}
Regarding the third integral in~\eqref{eq:Pf_Append_1}, we use~\eqref{eq:ap_hyp_inf} and that $p<2$ to obtain
\begin{equation}
\label{eq:Pf_Append_3}
\into |h(u_p)| \left(e^{\eta p |u_p|}-1 \right) |f| \leq \max_{|s|\leq M} |h(s)| e^{2M \eta} \|f\|_{L^1(\Omega)} + \sup_{|s|> M} |h(s)| \int_{\{|u_p|>M\}}\left(e^{\eta p |u_p|}-1 \right) |f|.
\end{equation}
Now, we focus our attention on the last term of~\eqref{eq:Pf_Append_3}. First, since we may assume $M>1$, we can apply the inequality~\eqref{eq:ap_elemental}. Then, we use H\"older inequalities in Lorentz spaces and, finally, the Sobolev embedding. In this way, we deduce
\begin{equation}
\label{eq:Pf_Append_4}
\begin{split}
\int_{\{|u_p|>M\}} \left(e^{\eta p |u_p|}-1 \right) |f| 
&\leq C_p \int_{\{|u_p|>M\}} \left(e^{\eta |u_p|}-1 \right)^p |f| 
\leq C_p \|f\|_{L^{N,\infty}(\Omega)} \left\|\left(e^{\eta |u_p|}-1 \right)^p\right\|_{L^{\frac{N}{N-1},1}(\Omega)}  \\
&\leq C_p \|f\|_{L^{N,\infty}(\Omega)} |\Omega|^\frac{p-1}{N} \left\|\left(e^{\eta |u_p|}-1 \right)^p\right\|_{L^{\frac{N}{N-p},1}(\Omega)}\\
&= C_p \|f\|_{L^{N,\infty}(\Omega)} |\Omega|^\frac{p-1}{N} \left\|e^{\eta |u_p|}-1 \right\|_{L^{\frac{Np}{N-p},p}(\Omega)}^p\\
& \leq C_p \|f\|_{L^{N,\infty}(\Omega)} |\Omega|^\frac{p-1}{N} \tilde{\mathcal{S}}_p^p \eta^p \into e^{\eta p |u_p|} |\nabla u_p|^p,
\end{split}
\end{equation}
where $\tilde{\mathcal{S}}_p$ is the best constant of the Sobolev embedding $\sob \hookrightarrow L^{p^*,p}(\Omega)$, see~\cite{Lorentz}. Combining the estimates \eqref{eq:Pf_Append_1}, \eqref{eq:Pf_Append_2}, \eqref{eq:Pf_Append_3} and~\eqref{eq:Pf_Append_4}, and rearranging some terms, we obtain
\begin{equation}
\label{eq:Pf_Append_5}
\left( \eta p - C_p \|f\|_{L^{N,\infty}(\Omega)} |\Omega|^\frac{p-1}{N} \tilde{\mathcal{S}}_p^p \eta^p \sup_{|s|> M} |h(s)| - \max_{|s|\leq M} g(s) \right) \into e^{\eta p |u_p|} |\nabla u_p|^p \leq \max_{|s|\leq M} |h(s)| e^{2M \eta} \|f\|_{L^1(\Omega)}.
\end{equation}

The constant that multiplies the integral in~\eqref{eq:Pf_Append_5} is positive if $\eta$ is large and $p$ is close to 1. To see this, take $\varepsilon>0$ with $\varepsilon < 1 - \|f\|_{L^{N,\infty}(\Omega)} \tilde{\mathcal{S}}_1 \sup_{|s|> M} |h(s)|$ (recall that~\eqref{eq:ap_M} holds). We fix $\eta\geq 1$ large such that
\begin{equation}
\label{eq:Pf_Append_6}
\varepsilon\eta - \max_{|s|\leq M} g(s) > 0.
\end{equation}
Since $C_p\to 1$ and $\tilde{\mathcal{S}}_p \to \tilde{\mathcal{S}}_1$ as $p\to 1^+$ (see~\cite{Lorentz}), then
\[
\lim_{p\to 1^+} \left( p - C_p \|f\|_{L^{N,\infty}(\Omega)} |\Omega|^\frac{p-1}{N} \tilde{\mathcal{S}}_p^p \eta^{p-1} \sup_{|s|> M} |h(s)|\right) = 1- \|f\|_{L^{N,\infty}(\Omega)} \tilde{\mathcal{S}}_1 \sup_{|s|> M} |h(s)|,
\]
so there exists some $p_0\in(1,2)$ such that
\begin{equation}
\label{eq:Pf_Append_7}
p - C_p \|f\|_{L^{N,\infty}(\Omega)} |\Omega|^\frac{p-1}{N} \tilde{\mathcal{S}}_p^p \eta^{p-1} \sup_{|s|> M} |h(s)| \geq \varepsilon,\ \forall p\in (1,p_0).
\end{equation}

Therefore, using~\eqref{eq:Pf_Append_7} and~\eqref{eq:Pf_Append_6}, we deduce from~\eqref{eq:Pf_Append_5} that
\begin{equation}
\label{eq:Pf_Append_7bis}
\into e^{\eta p |u_p|} |\nabla u_p|^p \leq C,\ \forall p\in (1,p_0),
\end{equation}
for some $C>0$. Finally, using the Young inequality, we obtain \begin{equation}
\label{eq:Pf_Append_8}
\into \left|\nabla \left( e^{\eta |u_p|}-1 \right) \right| 
\leq \frac{\eta^p}{p} \into e^{\eta p |u_p|} |\nabla u_p|^p + \frac{p-1}{p}|\Omega| \leq C,\ \forall p\in (1,p_0),
\end{equation}
where $C>0$. As a consequence of~\eqref{eq:Pf_Append_8}, the sequence $u_p$ is bounded in $BV(\Omega)$. Indeed, one has
\[
\eta \into |\nabla u_p| \leq \eta\into e^{\eta |u_p|}|\nabla u_p| = \into \left|\nabla \left( e^{\eta |u_p|}-1 \right) \right| 
\leq C,\ \forall p\in (1,p_0).
\]
In this way, there exists some $u\in\bvo$ and some subsequence of $u_p$, not relabelled, such that $u_p \to u$ in $L^1(\Omega)$ and $\nabla u_p$ $*$-weakly converges to $Du$ as measures in $\Omega$. Besides, since $e^{\eta |u_p|}-1$ is bounded in $L^1(\Omega)$, the Lebesgue Theorem implies that $u_p \to u$ in $L^q(\Omega)$ for any $q<+\infty$ as $p\to 1^+$.
\end{proof}

Next, we show that the sequence $u_p$ is bounded in $\lio$. Moreover, we give an explicit bound in $\lio$ for $u$.

\begin{lemma}
Under the assumptions of Theorem~\ref{th:Appendix}, the sequence $u_p$ of solutions to~\eqref{eq:Pb_pLap} is bounded in $\lio$. Furthermore, the limit function $u\in \bvo$ given by Lemma~\ref{lem:Appendix_BV} belongs to $\lio$ and satisfies
\[
\|u\|_\lio \leq M,
\]
where $M>0$ is given by~\eqref{eq:ap_M}.
\end{lemma}

\begin{proof}
 For $k>0$, we consider $G_k(s)$ defined as in \eqref{eq:def_Tk_Gk}. Given $k\geq M+p-1$, we take $G_k(u_p)\in W_0^{1,p}(\Omega)\cap L^\infty(\Omega)$ as test function in~\eqref{eq:Pb_pLap}, obtaining after using the H\"older and the Sobolev inequalities that
\begin{equation}
\label{eq:Pf_Append_9}
\begin{split}
    \int_\Omega |\nabla G_k(u_p)|^p &= \int_\Omega h(u_p) G_k(u_p) f \leq \sup_{|s|> M} |h(s)| \|f\|_{L^{N,\infty}(\Omega)} \|G_k(u_p)\|_{L^{\frac{N}{N-1},1}(\Omega)}\\
    & \leq \|f\|_{L^{N,\infty}(\Omega)} \tilde{\mathcal{S}}_1 \sup_{|s|> M} |h(s)| \int_\Omega |\nabla G_k(u_p)|.
\end{split}
\end{equation}
where we have taken into account that $g_p(s)=0$ when $|s|\geq M+p-1$. On the other hand, using Young's inequality, we deduce
\begin{align}
    \label{eq:Pf_Append_10}
    \int_\Omega |\nabla G_k(u_p)| \leq \int_\Omega |\nabla G_k(u_p)|^p + \frac{p-1}{p} |\{|u_p|>k\}|,
\end{align}
so~\eqref{eq:Pf_Append_9} and~\eqref{eq:Pf_Append_10} give
\begin{equation}
    \label{eq:Pf_Append_11}
    \int_\Omega |\nabla G_k(u_p)| \leq \|f\|_{L^{N,\infty}(\Omega)} \tilde{\mathcal{S}}_1 \sup_{|s|> M} |h(s)| \int_\Omega |\nabla G_k(u_p)| + \frac{p-1}{p} |\{|u_p|>k\}|.
\end{equation}
Since $M$ is such that~\eqref{eq:ap_M} holds, then one deduces from~\eqref{eq:Pf_Append_11} that
\begin{equation*}
    \int_\Omega |\nabla G_k(u_p)| \leq \frac{p-1}{p} \frac{|\{|u_p|>k\}|}{1-\|f\|_{L^{N,\infty}(\Omega)} \tilde{\mathcal{S}}_1 \sup_{|s|> M} |h(s)|},\ \forall k\geq M+p-1,
\end{equation*}
and, by the Sobolev inequality, we obtain
\begin{equation}
\label{eq:Pf_Append_12}
\|G_k(u_p)\|_{L^\frac{N}{N-1}(\Omega)} \leq  \frac{p-1}{p} \frac{ \mathcal{S}_1 |\{|u_p|>k\}|}{1-\|f\|_{L^{N,\infty}(\Omega)} \tilde{\mathcal{S}}_1 \sup_{|s|> M} |h(s)|},\ \forall k\geq M+p-1.
\end{equation}

In the following, we show that $u_p$ is uniformly bounded in $\lio$. Given $l>k$, we deduce from~\eqref{eq:Pf_Append_12} that
\[
|l-k| \, |\{|u_p|>l\}|^\frac{N-1}{N} \leq
\|G_k(u_p)\|_{L^\frac{N}{N-1}(\Omega)} \leq \frac{ \mathcal{S}_1 |\{|u_p|>k\}|}{1-\|f\|_{L^{N,\infty}(\Omega)} \tilde{\mathcal{S}}_1 \sup_{|s|> M} |h(s)|},\ \forall l>k\geq M+1,
\]
which is equivalent to
\begin{equation}
    \label{eq:Pf_Append_13}
     |\{|u_p|>l\}| \leq
    \left( \frac{\mathcal{S}_1}{1-\|f\|_{L^{N,\infty}(\Omega)} \tilde{\mathcal{S}}_1 \sup_{|s|> M} |h(s)|} \right)^\frac{N}{N-1} \frac{|\{|u_p|>k\}|^\frac{N}{N-1}}{|l-k|^\frac{N}{N-1}}, \ \forall l>k\geq M+1.
\end{equation}
Estimate~\eqref{eq:Pf_Append_13} allows us to apply standard Stampacchia's Lemma (\cite{Stamp}) to deduce that $u_p$ is uniformly bounded in $L^\infty(\Omega)$.

Now, observe that we can take $k=M+p-1$ in~\eqref{eq:Pf_Append_12} obtaining, for some $C>0$, that
\begin{equation}
\label{eq:Pf_Append_14}
\|G_{M+p-1}(u_p)\|_{L^\frac{N}{N-1}(\Omega)} \leq  \frac{p-1}{p} C.
\end{equation}
Since $G_{M+p-1}(u_p) \to G_M(u)$ in $L^\frac{N}{N-1}(\Omega)$ as $p\to 1^+$, one may take limits in~\eqref{eq:Pf_Append_14} to obtain that $G_M (u)=0$. This implies that $u\in\lio$ and $\|u\|_\lio \leq M$.
\end{proof}

We are now in position to demonstrate Theorem~\ref{th:Appendix}.

\begin{proof}[Proof of Theorem~\ref{th:Appendix}]
Although some parts of the proof are standard, we include them for completeness. We divide the proof into several steps.
\medskip 

\textbf{Step 1.} Existence of the vector field $z$.
\smallskip

Let $1\leq q <\frac{p}{p-1}$. Using the H\"older inequality and~\eqref{eq:Pf_Append_7bis}, we deduce
\begin{equation}
\label{eq:Pf_Append_15}
\left(\into \left||\nabla u_p|^{p-2} \nabla u_p\right|^q \right)^\frac{1}{q} \leq \left(\into |\nabla u_p|^p \right)^\frac{p-1}{p} |\Omega|^{\frac{1}{q}-\frac{p-1}{p}} \leq C^\frac{p-1}{p} |\Omega|^{\frac{1}{q}-\frac{p-1}{p}}. 
\end{equation}
It follows that, for any $q<+\infty$ fixed, the sequence $|\nabla u_p|^{p-2} \nabla u_p$ is bounded in $L^q(\Omega)^N$. Hence, up to subsequences, $|\nabla u_p|^{p-2} \nabla u_p$ converges weakly in $L^q(\Omega)^N$ to some $z_q\in L^q(\Omega)^N$. A standard diagonal argument then yields the existence of a vector field $z$, independent of $q$, such that $|\nabla u_p|^{p-2} \nabla u_p\rightharpoonup z$ weakly in $L^q(\Omega)^N$ for every $q<+\infty$. Letting $p\to 1^+$ in~\eqref{eq:Pf_Append_15} and using the weak lower semicontinuity of the norm, one obtains that 
\[
\|z\|_{L^q(\Omega)^N} \leq |\Omega|^\frac{1}{q},\ \forall q\in [1, +\infty).
\]
Finally, letting $q\to+\infty$, we deduce that $z\in \lio^N$ and $\|z\|_{\lio^N}\leq 1$.
\medskip

\textbf{Step 2.} The vector field $z$ is in $\DM(\Omega)$.
\smallskip

For any $s\in \R$, let $\Gamma(s) := \int_0^s g(t) \, \mathrm{d}t$ and $\Gamma_{p,p}(s) := \int_0^s g_p(t)^\frac{1}{p} \, \mathrm{d}t$. Taking a non-negative $\varphi\in C_c^1(\Omega)$ as test function in~\eqref{eq:Pb_pLap}, we obtain
\begin{equation}
\label{eq:Pf_Append_16}
\into |\nabla u_p|^{p-2} \nabla u_p \cdot \nabla \varphi = \into \varphi |\nabla \Gamma_{p,p}(u_p)|^p + \into h(u_p)f \varphi.
\end{equation}
Note that, thanks to the Young inequality, we have
\[
\into \varphi |\nabla \Gamma_{p,p}(u_p)| \leq \into \varphi |\nabla \Gamma_{p,p}(u_p)|^p + \frac{p-1}{p} \into \varphi
\]
and, since $\Gamma_{p,p}(u_p) \to \Gamma(u)$ in $L^1(\Omega)$ because $\|u\|_\lio\leq M$, then $\liminf_{p\to 1^+} \into \varphi |\nabla \Gamma_{p,p}(u_p)|^p \geq \into \varphi |D\Gamma(u)|$. Hence, passing to the limit in~\eqref{eq:Pf_Append_16} we obtain
\begin{equation*}
    -\operatorname{div} z \geq |D\Gamma(u)| + h(u)f \text{ in } \mathcal{D}'(\Omega).
\end{equation*}
Therefore, $z\in\DM_{\rm loc}(\Omega)$. Moreover, as $h(u)f\in L^{N,\infty}(\Omega)$, a slight modification of Lemma~\ref{lem:Global_z} gives that $z\in \DM(\Omega)$.  

\medskip
\textbf{Step 3.} Distributional formulation~\eqref{eq:def_distA}, identification of $z$ given by~\eqref{eq:def_pairA} and $D^ju=0$.
\smallskip

We define $\Gamma_{p,1}(s) := \int_0^s g_p(t) \, \mathrm{d}t$ for every $s\in\R$. Let $0\leq \varphi \in C_c^1(\Omega)$. Taking $e^{\Gamma_{p,1}(u_p)} \varphi \in \sob\cap \lio$ as test function in~\eqref{eq:Pb_pLap}, we deduce after cancellations that
\begin{equation}
\label{eq:Pf_Append_17}
\into e^{\Gamma_{p,1}(u_p)} |\nabla u_p|^{p-2} \nabla u_p \cdot \nabla \varphi = \into h(u_p) f e^{\Gamma_{p,1}(u_p)} \varphi.
\end{equation}
Since $e^{\Gamma_{p,1}(u_p)}$ is uniformly bounded in $\lio$, we can easily pass to the limit in~\eqref{eq:Pf_Append_17} to obtain
\begin{equation*}
\into e^{\Gamma(u)} z \cdot \nabla \varphi = \into h(u) f e^{\Gamma(u)}  \varphi,
\end{equation*}
that is,
\begin{equation*}
-\operatorname{div} \left( e^{\Gamma(u)}z \right) = h(u) f e^{\Gamma(u)} \text{ as measures in } \Omega.
\end{equation*}

At this point, we can follow line by line the proof of Theorem~\ref{th:Exist} (Steps 4--6) to obtain that $D^j u=0$, that the distributional formulation \eqref{eq:def_distA} holds and that $(z,Du)=|Du|$ as measures in $\Omega$.

\medskip
\textbf{Step 4.} Boundary condition~\eqref{eq:def_bordA}.
\smallskip

Due to the absence of singular terms, the boundary condition can be shown more easily than in Theorem~\ref{th:Exist}. We set for any $p\in (1,2)$ and any $s\in \R$ the functions
\begin{equation*}
\tilde \Gamma(s):= \int_0^s \left(m-tg(t)\right)|t|^{m-1}\, \mathrm{d}t \qquad \text{and} \qquad \tilde \Gamma_p(s):= \int_0^s \left(m-tg_p(t)\right)|t|^{m-1}\, \mathrm{d}t,
\end{equation*}
and we take $m$ sufficiently large (for instance, for $m>(M+1)\max_{|s|\leq M}g(s)$) to ensure that the integrand $\left(m-tg_p(t)\right)|t|^{m-1}$ is non-negative for any $t\in\R$.

 Let $\phi$ be a non-negative Lipschitz function on $\partial\Omega$. Then there exists $\varphi \in C(\overline{\Omega})$ which is a non-negative Lipschitz function on $\Omega$ and such that $\varphi\big|_{\partial\Omega} = \phi$.   Taking $|u_p|^{m-1} u_p \varphi \in \sob\cap \lio$ as test function in~\eqref{eq:Pb_pLap}, we obtain after rearranging some terms that
\begin{equation*}
\into \left(m - u_p g_p(u_p)\right) |u_p|^{m-1} \varphi |\nabla u_p|^p + \into |u_p|^{m-1} u_p |\nabla u_p|^{p-2} \nabla u_p \cdot  \nabla\varphi = \into |u_p|^{m-1} u_p h(u_p) f \varphi.
\end{equation*}
Taking into account that $\tilde \Gamma_p(u_p)$ has zero boundary trace and using the Young inequality, we deduce that
\begin{equation}
\label{eq:Pf_Append_18}
\begin{split}
\into & \varphi |\nabla \tilde \Gamma_p(u_p)| + \intdo |\tilde\Gamma_p (u_p) | \varphi  \ \dH\\
& =
\into \left(m - u_p g_p(u_p)\right) |u_p|^{m-1} \varphi |\nabla u_p|\\
&\leq \into \left(m - u_p g_p(u_p)\right) |u_p|^{m-1} \varphi |\nabla u_p|^p + \frac{p-1}{p}\into \left(m - u_p g_p(u_p)\right) |u_p|^{m-1} \varphi\\
&= \into |u_p|^{m-1} u_p h(u_p) f \varphi - \into |u_p|^{m-1} u_p |\nabla u_p|^{p-2} \nabla u_p \cdot  \nabla\varphi + \frac{p-1}{p}\into \left(m - u_p g_p(u_p)\right) |u_p|^{m-1} \varphi.
\end{split}
\end{equation}

Since $u_p$ is uniformly bounded in $\lio$, we can easily pass to the limit in the right-hand side of~\eqref{eq:Pf_Append_18}. In the left-hand side, we use the lower semicontinuity of the functional $I(v) = \into |Dv|\varphi + \intdo |v| \varphi$ (defined in $BV(\Omega)$) with respect to the $L^1(\Omega)$-norm. Thus, letting $p\to 1^+$ in~\eqref{eq:Pf_Append_18} we obtain that 
\begin{equation*}
\into \varphi |D\tilde \Gamma(u)| + \intdo |\tilde \Gamma(u)| \phi \ \dH \leq \into h(u) |u|^{m-1}u f \varphi - \into |u|^{m-1}u\, z  \cdot \nabla \varphi.
\end{equation*}

From this point, one may reason as in the proof of Theorem~\ref{th:Exist} to obtain that
\begin{equation*}
\intdo \phi \left(|\tilde\Gamma(u)| + |u|^{m-1}u [z,\nu] \right) \ \dH = \intdo \phi |u|^{m-1} \left((|u| + u[z,\nu]) - \frac{\int_0^u g(t)|t|^m \, \mathrm{d}t}{|u|^{m-1}} \right) \ \dH \leq 0
\end{equation*}
for every non-negative $\phi\in C(\partial\Omega)$ (by an easy  density argument), where we have taken into account that $|\tilde \Gamma(s)| = |s|^m - \int^s_0 g(t)|t|^m\, \mathrm{d}t$. Therefore, for $\mathcal{H}^{N-1}$-$\ae$ $x\in\partial\Omega$, it holds that either $u(x)=0$ or
\begin{equation}
\label{eq:Pf_Append_19}
0\leq |u(x)| + u(x)[z,\nu](x) \leq \frac{\int_0^{u(x)} g(t) |t|^m\, \mathrm{d}t}{|u(x)|^{m-1}},
\end{equation}
as $\|[z,\nu]\|_{L^\infty(\partial\Omega)} \leq 1$. If~\eqref{eq:Pf_Append_19} holds, since $g>0$,  then $u(x)\ge 0$ and
\begin{equation*}
0\leq u(x) + u(x)[z,\nu](x) \leq \frac{u(x)^2}{m+1} \max_{|s|\leq M} g(s).
\end{equation*}
Taking the limit as $m\to\infty$ we obtain $u(x) + u(x)[z,\nu](x)=0$. This concludes the proof.
\end{proof}

\section*{Acknowledgements}
The first author has been funded by Junta de Andaluc\'ia (grant FQM-194),
by the Spanish Ministry of Science and Innovation, Agencia Estatal de Investigaci\'on (AEI)
and Fondo Europeo de Desarrollo Regional (FEDER) (grant PID2021-122122NB-I00) and by
the FPU predoctoral fellowship of the Spanish Ministry of Universities (FPU21/04849).
The second and the third author  have been  partially supported by the Gruppo Nazionale per l’Analisi Matematica, la Probabilità e le loro Applicazioni (GNAMPA) of the Istituto Nazionale di Alta Matematica (INdAM).

\section*{Conflict of interest declaration}

The authors declare no competing interests.

\section*{Data availability statement}
 We do not analyse or generate any datasets, because our work
proceeds within a theoretical and mathematical approach. One
can obtain the relevant materials from the references below.  

\section*{AI Statement}   The authors used Claude Opus 5.5 (Anthropic) to assist with the detection of typographical errors, consistency checks, and the verification of bibliographic references. All mathematical arguments and proofs were developed, written, and verified by the authors, who take full responsibility for the entire content of the manuscript.

\end{document}